\documentclass[reqno,a4paper]{amsart}

\usepackage{etex}

\usepackage{amsfonts,amsmath,amssymb,amsthm}
\usepackage{alltt,esvect,graphicx,stmaryrd,tikz,xcolor}
\usepackage{pgffor}
\usepackage[arrow, matrix, curve]{xy}
\usepackage[titletoc]{appendix}
\usepackage[linesnumbered]{algorithm2e}
\usepackage[T1]{fontenc}
\usepackage[utf8]{inputenc}
\usepackage{helvet}
\usetikzlibrary{calc,intersections,through,shapes,arrows}
\usepackage{hyperref}
\usepackage{mathtools}
\usepackage{subfig}
\usepackage{url}
\usepackage{booktabs}
\usepackage{comment}
\usepackage[textsize=tiny]{todonotes}

\usepackage{enumitem} 
\usetikzlibrary{cd}
\usetikzlibrary{patterns}
\usetikzlibrary{math}
\usetikzlibrary{shapes}
\usetikzlibrary{decorations.fractals}

\usepackage{xstring}

\newtheorem{theorem}{Theorem}[section]
\newtheorem{lemma}[theorem]{Lemma}
\newtheorem{proposition}[theorem]{Proposition}
\newtheorem{corollary}[theorem]{Corollary}

\theoremstyle{definition}
\newtheorem{definition}[theorem]{Definition}
\newtheorem{example}[theorem]{Example}

\newtheorem{remark}[theorem]{Remark}

\definecolor{colorX}{RGB}{240,0,0}
\definecolor{colorY}{RGB}{0,0,160}
\definecolor{colorZ}{RGB}{0,160,0}

\newcommand{\N}{\mathbb{N}}
\newcommand{\Z}{\mathbb{Z}}

\newcommand{\R}{\mathbb{R}}
\newcommand{\C}{\mathbb{C}}

\renewcommand{\P}{\mathbb{P}}
\newcommand{\CI}{\mathcal{I}}
\newcommand{\CF}{\mathcal{F}}
\newcommand{\CL}{\mathcal{L}}

\newcommand{\CE}{\mathcal{E}}
\newcommand{\CH}{\mathcal{H}}
\newcommand{\CO}{\mathcal{O}}

\newcommand{\CX}{\mathcal{X}}
\newcommand{\CY}{\mathcal{Y}}
\newcommand{\CZ}{\mathcal{Z}}

\newcommand{\Hom}{\mathop{\rm Hom}\nolimits}
\newcommand{\Ext}{\mathop{\rm Ext}\nolimits}
\newcommand{\Id}{\mathop{\rm Id}\nolimits}
\newcommand{\Pic}{\operatorname{Pic}}
\newcommand{\gH}{\operatorname{H}}

\renewcommand{\ker}{\mathop{\rm ker}\nolimits}

\newcommand{\rk}{\mathop{\rm rk}\nolimits}

\newcommand{\Proj}{\mr{Proj}\,}
\newcommand{\Sym}{\operatorname{Sym}}

\newcommand{\mr}{\mathrm}
\newcommand{\mb}{\mathbf}
\newcommand{\mc}{\mathcal}

\newcommand{\bp}{\begin{para}}
\newcommand{\ep}{\end{para}}
\newcommand{\bps}{\begin{paras}}
\newcommand{\eps}{\end{paras}}
\newcommand{\benum}{\begin{enumerate}[{\rm(i)}]}
\newcommand{\eenum}{\end{enumerate}}

\newcommand{\kst}{\,|\;}

\newcommand{\kss}{\scriptscriptstyle}
\newcommand{\kbb}{{\kss \bullet}}

\newcommand{\Eff}{\operatorname{Eff}}
\newcommand{\Nef}{\operatorname{Nef}}

\newcommand{\nImm}{\operatorname{\CI}}
\newcommand{\helixL}{\operatorname{h_L}}
\newcommand{\helixR}{\operatorname{h_R}}
\newcommand{\helex}{\hslash}
\newcommand{\DeltaUp}{\Delta^{\hspace{-0.1em}\mbox{\tiny up}}}

\newcommand{\mes}{{maximal exceptional sequence}}
\newcommand{\cl}{{s}}

\newcommand{\xa}{\alpha} 
\newcommand{\xb}{\beta}

\newcommand{\nAcyc}{{\operatorname{Acyc}}}

\newcommand{\CD}{{\mathcal D}}
\newcommand{\CT}{{\mathcal T}}

\newcommand{\cX}{{\CX}}
\newcommand{\cY}{{\CY}}
\newcommand{\cZ}{{\CZ}}
\newcommand{\cA}{{A}}
\newcommand{\cB}{{B}}

\newcommand{\lH}{{\ell}}
\newcommand{\lV}{{r}}
\newcommand{\tO}{{O}}
\newcommand{\pO}{{P}}

\newcommand{\free}{{\operatorname{K_{free}}}}
\newcommand{\Zbound}{{\CZ_{-\infty}}}

\newcommand{\Zfree}{{\CZ_{\operatorname{free}}}}

\newcommand{\layers}{{\operatorname{K_{layers}}}}
\newcommand{\pF}{F}

\newcommand{\drawExSec}[1]{
 \StrSubstitute{#1}{,}{/}[\replacedSlash]
 \StrSubstitute{\replacedSlash}{;}{,}[\xycoords]
 \foreach [count=\i] \x/\y in \xycoords
 {
  \node[shape=circle,draw,fill=white,inner sep=1pt] at (\x,\y) {$\scriptstyle \i$};
 }
}
\newcommand{\drawExSet}[1]{
 \StrSubstitute{#1}{,}{/}[\replacedSlash]
 \StrSubstitute{\replacedSlash}{;}{,}[\xycoords]
 \foreach [count=\i] \x/\y in \xycoords
 {
  \ifthenelse {1=\i}
   {\fill (\x,\y) circle (5pt); \draw (\x,\y) circle (7pt);}
   {\fill (\x,\y) circle (5pt);}
 }
}
\definecolor{oliwkowy}{HTML}{627037}
\definecolor{lightblue}{RGB}{135,206,250}
\definecolor{darkblue}{RGB}{0,0,160}
\definecolor{darkgreen}{RGB}{0,160,0}
\definecolor{veryPeri}{RGB}{102,103,171}
\definecolor{intOrange}{rgb}{1.0,.310,.0}
\definecolor{MidnightBlue}{RGB}{102,103,171}

\definecolor{cklaus}{rgb}{0.0,.288,.378}
\definecolor{candreas}{RGB}{135,206,250}

\newcounter{para}[section]
\newenvironment*{para}[1]{\refstepcounter{para}\noindent\ignorespaces{\bf\thepara.~#1.}}{\ignorespacesafterend\bigskip}
\newenvironment*{paras}[1]{{\bf #1.}}{\ignorespacesafterend\bigskip}
\numberwithin{para}{section}

\newcommand{\Imm}{{\operatorname{Imm}}}

\setlist[enumerate,1]{label = (\roman*),ref = \theenumii.\roman*}

\newcommand{\SymE}{\Sym^\bullet\CE^\vee}

\newcommand{\Tang}{\mc T}

\newcommand{\Gr}{\mathbb{G}r}
\newcommand{\Flag}{\mathbb{F}l}

\newcommand{\leftMut}{\mb L}
\newcommand{\rightMut}{\mb R}
\newcommand{\sod}[1]{\langle #1 \rangle}
\newcommand{\cone}[1]{\langle #1 \rangle}

\newcommand{\thmExSec}{A}
\newcommand{\thmStrongExSec}{B}
\newcommand{\thmFull}{C}

\newcommand{\RGamma}{{\R\Gamma}}
\newcommand{\RHom}{{\R\!\operatorname{Hom}}}
\newcommand{\kk}{{\operatorname{k}}}
\newcommand{\eslb}{{exceptional sequence of line bundles}}
\newcommand{\meset}{{maximal exceptional set}}
\newcommand{\eff}{\mathrm{eff}}
\newcommand{\nef}{\mathrm{nef}}
\newcommand{\effec}{{effective}}
\newcommand{\into}{\hookrightarrow}
\newcommand{\onto}{\twoheadrightarrow}
\newcommand{\arXiv}[1]{\href{http://arxiv.org/abs/#1}{\textup{\texttt{arXiv:#1}}}}

\newif\ifnotext

\begin{document}
\parindent0mm

\title[Exceptional sequences]
{Exceptional sequences of line bundles on projective bundles}

\author[K.~Altmann]{Klaus Altmann
}
\address{Institut f\"ur Mathematik,
FU Berlin,
K\"onigin-Luise-Str.~24-26,
D-14195 Berlin
}
\email{altmann@math.fu-berlin.de}
\author[A.~Hochenegger]{Andreas Hochenegger}
\address{
Dipartimento di Matematica ``Francesco Brioschi'',
Politecnico di Milano,
via Bonardi 9,
I-20133 Milano 
}
\email{andreas.hochenegger@polimi.it}
\author[F.~Witt]{Frederik Witt}
\address{
Fachbereich Mathematik,
U Stuttgart,
Pfaffenwaldring 57,
D-70569 Stuttgart
}
\email{witt@mathematik.uni-stuttgart.de}

\begin{abstract}
For a vector bundle $\CE\to\P^\lH$ we investigate exceptional sequences of line bundles on the total space of the projectivisation $X=\P(\CE)$. In particular, we consider the case of the cotangent bundle of $\P^\lH$. If $\lH=2$, we completely classify the (strong) exceptional sequences and show that any maximal exceptional sequence is full. For general $\lH$, we prove that the Rouquier dimension of $\CD(X)$ equals $\dim X$, thereby confirming a conjecture of Orlov.
\end{abstract}

\maketitle

\section{Introduction}
\label{sec:Intro}
Let $X$ be a smooth projective variety over an algebraically closed field $\kk$ of characteristic $0$ such that 
\[
\RGamma(X,\CO_X)=\kk\quad\mbox{and}\quad\rk K_0(X)<\infty;
\]
this will hold for all varieties considered in this paper. A sequence of line bundles\footnote{By abuse of language we always speak of line and vector bundles even if invertible or locally free sheaves would sometimes be more appropriate.} $\CL_1,\ldots,\CL_n$ is {\em exceptional} if there are no nontrivial backwards Ext-groups, i.e., 
\[
\Ext^k(\CL_j,\CL_i)=0\hspace{0.8em}\mbox{whenever }i<j.
\]
Since $\Ext^k(\CL_j,\CL_i)=H^k(X,\CL_j^{-1}\otimes\CL_i)$, this boils down to require that the line bundles $\CL_j^{-1}\otimes\CL_i$ belong to the {\em immaculate locus}
\[
\Imm(X):=\{\CL\in\Pic(X)\mid \RGamma(X,\CL)=0\} 
\]
of cohomologically trivial line bundles.\footnote{Here and in the sequel, we identify line bundles with their classes in the Picard group.} The length $n$ of an \eslb\ is bounded by the rank of $K_0(X)$. We say that the sequence is {\em maximal} if actually $n=\rk K_0(X)$. It is useful to think of an exceptional sequence in terms of a subset $\cl\subseteq\Pic(X)$ for which exists an (in general non-unique) {\em exceptional order} $<_\tO$, that is, a total order turning $\cl$ into an exceptional sequence. 

\medskip

In~\cite{fesfano}, the first and third author discuss a combinatorial approach to the structure of maximal exceptional sequences on smooth projective toric varieties of Picard rank two. By results of Kleinschmidt~\cite{kle88}, these varieties arise as the total space of the projectivised bundle $\P(\bigoplus_{j=0}^\lV\CL_i)\to\P^\lH$ given by a direct sum of line bundles. We subsequently refer to this as the {\em toric case}. 

\medskip

The goal of this sequel is threefold: First, to put the combinatorial arguments into a more geometric context. Second, to extend the methods past toric varieties. Thirdly, to discuss applications involving strongly exceptional sequences. For this we consider the family of projective bundles $X_\lH=\P(\Omega_{\P^\lH})\to\P^\lH$ given by the cotangent bundle of $\P^\lH$ (for technical reasons, we actually consider the twisted bundle $\Omega_{\P^\lH}(1)=\Omega_{\P^\lH}\otimes\CO_{\P^\lH}(1)$ which, however, induces the same projectivised bundle). In this context, we set out to prove the following model theorems.

\medskip

{\bf Theorem~\thmExSec.} 
{\em Classification of the \meset s of $X$.}

\smallskip

The statement is deliberately vague. In the best case, classification means enumeration of any \mes, which for $\lH=2$ is achieved in Corollary~\ref{coro:FullList}. Less ambitiously, we can settle for a classification of the underlying \meset s as in~\cite{fesfano}, which classifies \meset s for the toric case in this weaker sense. For applications such as Theorem~\thmFull\ below, a weak classification of \mes s up to {\em helixing} is already sufficient. 

\smallskip

In principle, our methods also extend to higher $\lH$ at least for a classification of the \meset s, but unlike the toric case we failed to find a general pattern which is why we restrict to $\lH=2$. 

\medskip

{\bf Theorem~\thmStrongExSec.} 
{\em Classification of strongly exceptional \meset s of $X$.}

\smallskip

Strongly exceptional sets enjoy the additional property of
\[
\Ext^k(\CL,\CL')=0\text{ for all $\CL$, $\CL'\in\cl$ and all $k>0$}.  
\]
Homologically, they are interesting for their connections with tilting bundles, see Section~\ref{sec:ExObMut}. Combinatorially, they are interesting because they are {\em \effec} which roughly speaking means that they are the most flexible sets for finding exceptional orders, cf.\  Subsection~\eqref{subsec:Div}. We prove Theorem~\thmStrongExSec\ for the toric case and $X_2=\P(\Omega_{\P^2})$. Here, we classify in the weak sense, that is, we explicitly determine the sets of Theorem~\thmExSec\ underlying a strongly exceptional set.

\medskip

{\bf Theorem~\thmFull.}
{\em Every \meset\ of line bundles is full.}

\smallskip

A maximal exceptional sequence is {\em full} if as a set it generates $\CD(X)$. In general, this is false due to the existence of phantom categories, see~\cite{Phantom}. On the other hand, Kuznetsov conjectured in his ICM '14 talk that if there exists at least one full maximal exceptional sequence (not necessarily of line bundles), then any other maximal sequence is also full~\cite{KuznetsovICM}. 

\smallskip

Though Krah exhibits a counterexample in~\cite{Krah}, this property holds for exceptional sequences of line bundles on $X_2$, cf.\ Theorem~\ref{thm:CforX2} (for the toric case, see~\cite{fesfano}). The idea consists in exhibiting a suitable system of representatives of \meset s up to helixing and to reduce these to sequences of {\em Orlov-type} which we discuss in Section~\ref{sec:ExObMut}. We note in passing that \cite[Thm. 4.14]{borisov21} shows that every strongly exceptional \meset\ of maximal length is already full on a Fano toric Deligne-Mumford stack with Picard rank at most two. 

\medskip

As an aside on the theme of strongly exceptional sequences and tilting bundles we shall also prove that the Rouquier dimension of $\CD(X_\lH)$ equals $\dim X_\lH$ for $\lH\geq2$, thereby lending further support for a conjecture by Orlov~\cite{Orlov-generators}. The same method also applies to the toric case provided the anti-canonical line bundle $-K_X$ is nef. In fact, Orlov's conjecture holds for any toric variety, see~\cite{RouquDim}.

\subsubsection*{Plan of the paper}
Section~\ref{sec:ExObMut} reviews some background on exceptional sequences. Section~\ref{sec:ExPosets} investigates the intersection of all possible exceptional orders on a given exceptional set, thereby deriving a criterion for strong exceptionality. Section~\ref{sec:ProjBund} discusses the general setup of projective bundles $\P(\CE)\to\P^\lH$, while Sections~\ref{sec:toric} and~\ref{sec:CoTang} investigate the toric case and $X_\lH$, respectively.
Finally, we gather some results on the Rouquier dimension in Section~\ref{sec:Rouquier}.

\subsubsection*{Acknowledgements}
We thank the anonymous referee for several very valuable remarks.

\section{Preliminaries}
\label{sec:ExObMut}
We recall some background on exceptional objects and related notions. A good reference is~\cite{Rudakov}. 

\medskip

Let $X$ be a smooth projective variety over an algebraically closed field $k$ and $\CD(X)$ the bounded derived category of coherent sheaves on $X$. Given a functor $F$, $\R\mr{F}$ denotes, if it exists, its right derived functor.

\begin{definition}
An object $E$ in $\CD(X)$ is called {\em exceptional} if $\RHom(E,E)=\kk$. A pair $(E,F)$ of exceptional objects is {\em exceptional} if in addition $\RHom(F,E)=0$. 
\end{definition}

\begin{definition}
For an exceptional object $E$ we define\footnote{We ignore the technicalities behind defining a functor as a cone; they are of no importance here, as we apply these functors only to objects.}
the {\em left mutation functor} $\leftMut_E$ on $\CD(X)$ by the triangle of functors
\[
\RHom(E,-) \otimes E \xrightarrow{\mathrm{ev}} \Id \to \leftMut_E.
\]
Dually, we have a {\em right mutation functor} $\rightMut_E$ given by 
\[
\rightMut_E \to \Id \xrightarrow{\mathrm{coev}} \RHom(-,E)^\vee \otimes E.
\]
\end{definition}

\medskip

Let $E$ be an exceptional object in $\CD(X)$. We put
\[
E^\perp=\{F\in\CD(X)\mid\RHom(E,F)=0\}\hspace{5pt}\mbox{and}\hspace{5pt}{}^\perp E=\{F\in\CD(X)\mid\RHom(F,E)=0\}. 
\]

\begin{proposition}
The left mutation functor of an exceptional object $E$ satisfies
\[
\leftMut_E(E) = 0 
\quad\text{and}\quad
\leftMut_E \colon E^\perp \xrightarrow{\sim} {}^\perp E.
\]
Similarly, the right mutation functor satisfies
\[
\rightMut_E(E) = 0
\quad\text{and}\quad
\rightMut_E \colon {}^\perp E \xrightarrow{\sim} E^\perp.
\]
\end{proposition}

\begin{example}
Since we assume $\RGamma(X,\CO_X)=\kk$, any line bundle on $X$ considered as an element in $\CD(X)$ is exceptional. Since for two line bundles $\mc L$ and $\mc L'$,
\[
\RHom(\CL',\CL)=\Ext^\kbb(\CL',\CL)=\mathrm{H}^\kbb(X,\CL'^{-1}\otimes\CL),
\]
$(\mc L,\mc L')$ defines an exceptional pair if and only if $\CL'^{-1}\otimes\CL$ 
lies in the so-called {\em immaculate locus} 
\[
\Imm:=\Imm(X):=\{\CL\in\Pic(X)\mid\gH^k(X,\CL)=0
\mbox{ for all }k\geq 0\}\subseteq\Pic(X) 
\]
of cohomologically trivial line bundles. 
\end{example}

Let $(E,F)$ be an exceptional pair. The pairs obtained by {\em left} and {\em right mutations}, namely $(\leftMut_E F,E)$ and $(F,\rightMut_F E)$, are again exceptional and generate the same triangulated subcategory as $(E,F)$. Moreover, the left and right mutations are inverse to each other in the sense that
\[
\rightMut_E \leftMut_E F = F
\quad\text{and}\quad
\leftMut_F \rightMut_F E = E.
\]

\begin{remark}
\label{rem:MutationLB}
$\leftMut_E$ and $\rightMut_F$ have a particularly simple expression 
if both pairs $(E,F)$ and $(F,E)$ are orthogonal, that is, $\RHom(E,F)=0= \RHom(F,E)$. In this case $\leftMut_E F = F$ and $\rightMut_F E = E$ so that mutation just interchanges $E$ and $F$.
\end{remark}

\begin{definition}
A sequence $E_1,\ldots,E_n$ of exceptional objects is {\em exceptional} if
\[
\RHom(E_j,E_i)=0\hspace{0.7em}\mbox{for all }j>i.  
\]
An exceptional sequence is {\em full} if it generates $\CD(X)$.
\end{definition}

\begin{example} 
\label{exam:Beilinson}
On $\P^n$, we have the full exceptional sequence of line bundles
\[
\big(\CO,\,\CO(1),\ldots,\CO(n)\big),
\]
the so-called {\em Beilinson sequence}~\cite{beilinson}.
\end{example}

One can think of exceptional sequences in terms of linearly independent subsets of the Grothendieck group; a full sequence yields an identification $K_0(X)\cong\oplus_{i=1}^n\Z E_i$. In particular, the length of any exceptional sequence is bounded by $\rk K_0(X)$. The sequence is {\em maximal} if equality holds. Therefore, fullness implies maximality while the converse is false in general.

\begin{proposition}
Let $E_1,\ldots,E_n$ be exceptional objects in $\CD(X)$, and $1\leq i < n$.
Then $(E_1,\ldots,E_n)$ is a (full) exceptional sequence if and only if
\[
(E_1,\ldots,E_{i-1},\leftMut_{E_i} E_{i+1}, E_i, E_{i+2},\ldots,E_n)
\]
is a (full) exceptional sequence. An analogous statement holds for right mutations.
\end{proposition}

\begin{remark}
\label{rem:Mut}
Mutating an exceptional sequence of vector bundles yields honest complexes in $\CD(X)$ in general. There are, however, a few cases of interest where mutation preserves invertibility or at least local freeness up to a shift. Let $(\CO_X,\CL)$ be an exceptional pair of line bundles, that is, $\CL^{-1}$ is immaculate. 
\begin{enumerate}
\item If $\CL$ is also immaculate, then $\CO_X$ and $\CL$ are mutually orthogonal, so (left) mutation means just interchanging the line bundles.
\item If, on the other hand, $\CL$ is globally generated with $\RGamma(X,\CL) = \mr{H}^0(X,\CL)$, then the triangle defining left mutation becomes the short exact sequence
\[
0 \to\leftMut_{\CO_X}\CL[-1]\to\mr{H}^0(X,\CL)\otimes\CO_X\to\CL\to0
\]
so that $\leftMut_{\CO_X} \CL[-1]$ is a vector bundle of rank $\dim H^0(X,\CL)-1$.
\end{enumerate}
\end{remark}

\begin{definition}
Let $\cl=(E_1,\ldots,E_n)$ be an exceptional sequence in $\CD(X)$, and let $K_X$ be the canonical line bundle on $X$. The operation on sequences given by
\[
\helixL(\cl):=(E_n \otimes K_X, E_1,\ldots,E_{n-1})
\]
is called {\em helixing to the left}. Similarly, the sequence
\[
\helixR(\cl):=(E_2,\ldots,E_n, E_1 \otimes K_X^{-1})
\]
is obtained by {\em helixing to the right}.
\end{definition}

\begin{remark}
This definition is slightly non-standard, as one would rather use the Serre functor $S_X=-\otimes K_X[\dim(X)]$. However, we continue to ignore possible shifts to keep exposition tight.
\end{remark}

\begin{proposition}
Exceptional objects $E_1,\ldots,E_n$ in $\CD(X)$ define a (full) exceptional sequence $(E_1,\ldots,E_n)$ if and only if helixing to the left (or right) yields a (full) exceptional sequence.
\end{proposition}

\begin{remark}
For full exceptional sequences, helixing is actually nothing else 
than a special instance of mutating: If $(E_1,\ldots,E_n)$ is a full exceptional sequence, then $E_n \otimes K_X$ and $\leftMut_{E_1} \circ \ldots \circ \leftMut_{E_{n-1}} E_n$ differ only by a shift. An analogous statement holds for right mutations.
\end{remark}

A {\em semi-orthogonal decomposition} of $\CD(X)$ is a collection $\CT_1,\ldots,\CT_n$ of full triangulated subcategories generating $\CD(X)$ such that $\Hom_{\CD(X)}(T_j,T_i) = 0$ for all $1\leq i<j\leq n$ and objects $T_i\in\CT_i$, $T_j\in\CT_j$. For instance, a full exceptional sequence $(E_1,\ldots,E_n)$ gives rise to the semi-orthogonal decomposition $\CT_1,\ldots,\CT_n$ of $\CD(X)$ where $\CT_i=\langle E_i\rangle$ is the triangulated category generated by $E_i$.

\begin{example}
\label{exam:OrlovType}
Consider a vector bundle $\CE\to Y$ of rank $\rk(\CE) = \lV+1$ with associated projective bundle
\[
\pi\colon X=\P(\CE):=\Proj(\SymE)\to Y  
\]
and relative hyperplane section $H$ determined by $\pi_*\CO(H)=\CE^\vee$. By~\cite{orlov92}, we have the semi-orthogonal decomposition 
\[
\CD(X) = \sod{ \pi^* \CD(Y), \pi^* \CD(Y) \otimes \CO(H), \ldots, \pi^* \CD(Y) \otimes \CO(\lV H)}
\]
given by $\CT_i=\langle\pi^* \CD(Y) \otimes \CO(i H)\rangle$, $i=0,\ldots,r$. For instance, let $(F_{i,0},\ldots,F_{i,\lH})$, $0 \leq i \leq \lV$, be (full) exceptional sequences on $Y$. Then 
\begin{equation}
\label{eq:OrlovType}
(\pi^* F_{0,1}, \ldots, \pi^* F_{1,\lH}, \ldots, \pi^* F_{\lV,1} \otimes \CO(\lV H), \ldots, \pi^* F_{\lV,\lH} \otimes \CO(\lV H))
\end{equation}
defines a (full) exceptional sequence on $X$. 
\end{example}

\begin{definition}
An exceptional sequence on $X=\P(\CE)$ is of {\em Orlov type} if it has the form~\eqref{eq:OrlovType}. 
\end{definition}

\begin{definition}
A {\em strongly exceptional} sequence is an exceptional sequence which in addition satisfies
\[ 
\R^k\!\Hom(E_j,E_i)=\Hom_{\CD(X)}(E_j,E_i[k])=0\mbox{ for all $k>0$ and $i$, $j$.} 
\] 
\end{definition}

One reason for studying full and strongly exceptional sequences is that any such sequence $(E_1,\ldots,E_n)$ gives rise to the {\em tilting object} $T = \bigoplus_i E_i$ of $\CD(X)$. By definition, this means that $\Hom_{\CD(X)}(T,T[i]) = 0$ for $i \neq 0$. Further, $\CD(X)$ is the smallest triangulated subcategory containing $T$ which is full and thick, i.e., closed under taking direct summands.

\begin{remark}
Note that helixing of an exceptional sequence will not preserve strongness in general. We call a full and strongly exceptional sequence $(E_1,\ldots,E_m)$ {\em strongly cyclic} if it does. (This definition is not completely standard as we assume fullness.)
\end{remark}

\section{Exceptional posets}
\label{sec:ExPosets}
From now on, the term exceptional set / sequence will always refer to an exceptional set / sequence of line bundles. 

\medskip

We study the combinatorics of an exceptional set via its poset of exceptional orders. In particular, we will derive a criterion for strong exceptionality.

\subsection{The poset associated with an exceptional set}
\label{subsec:definePosetPs}
We call a subset $\cl\subseteq\Pic(X)$ {\em exceptional} if there exists a total order $<_\tO$ on $\cl$ turning $\cl$ into an exceptional sequence. Any such order will be referred to as {\em exceptional}. 

\medskip

For an exceptional set $\cl$ we define its {\em associated poset}
\[
\pO=\pO(\cl):=\bigcap\{\tO\mid\tO\mbox{ is an exceptional order on }\cl\}\cup\{(\cA,\cA)\mid\cA\in\cl\}
\subseteq\cl\times\cl
\]
which induces a partial order $\leq_\cl$ on $\cl$. To describe $\leq_\cl$ in terms of a generating set, we consider
\[
\pF=\pF(\cl):=\{(A,B)\in\cl\times\cl\kst
A-B\not\in-\Imm(X)\}\subseteq\cl\times\cl.
\]
Since $A-B\not\in-\Imm(X)$ implies $A\leq_{\tO}B$ for any exceptional order $\tO$ on $\cl$, $\pF\subseteq\pO$. We write $\cA\leq_{\pF}\cB$ if $(\cA,\cB)\in\pF$ and $\langle\pF\rangle$ for the poset given by the {\em transitive hull} of $\pF$ inside $\pO$. By definition, this is the set consisting of $\pF$ and all pairs $(\cA,\cB)\in\cl\times\cl$ such that there exist $C_1,\ldots,C_m\in\cl$ with $\cA\leq_FC_1\leq_FC_2\leq_F\cdots\leq_FC_m\leq_F\cB$.

\begin{proposition}
\label{prop:pFpO}
Let $\cl$ be an exceptional set and $\pF=\pF(\cl)$. 
\begin{enumerate}
\item[\rm (i)] A total order $\tO$ on $\cl$ is exceptional if and only if
$\pF\subseteq\tO$.
\item[\rm (ii)] $\langle\pF\rangle=\pO$.
\end{enumerate}
\end{proposition}

\begin{proof}
(i) By design, $\pF\subseteq\tO$ for any exceptional order $\tO$ on $\cl$. Conversely, assume that $\pF\subseteq\tO$ for some total order $\tO$ on $\cl$. If $\tO$ were not exceptional, then there would exist $A$, $B\in\cl$ with $A<_\tO B$, but $B-A\not\in-\Imm(X)$. Hence $(B,A)\in\pF(\cl)\subseteq\tO$, that is, $B<_\tO A$, contradiction.

\medskip

(ii) Assume that $(A,B)\notin\langle\pF\rangle$. In view of Lemma~\ref{lem:ExtendingPosets} below, there exists a total order $\tO$ containing $\langle\pF\rangle$ and $(B,A)$. 
In particular, $(A,B)\notin\tO$. As $\tO$ is exceptional by (i), $(A,B)\notin\pO$ whence $\pO\setminus\langle\pF\rangle=\varnothing$.
\end{proof}

As we could not find a suitable reference we include a proof of the following folklore

\begin{lemma}[Extension of partial orders]
\label{lem:ExtendingPosets}
Let $P\subseteq X\times X$ be a partial order on a set $X$ and assume that for some $x\not=x'\in X$, $(x',x)\notin P$. Then there exists a partial order $P'$ with $(x,x')\in P'$ which extends $P$.
\end{lemma}

\begin{proof}
Let $Q:=P\cup\{(x,x')\}$ and write $q\leq_Qq'$ if $(q,q')\in Q$. Let $P':=\langle Q\rangle$ be the transitive hull of $Q$. Since $P$ is a partial order on $X$, we have $(x,x)\in P\subseteq P'$ for all $x\in X$.

\medskip

It remains to check the anti-symmetry of $P'$. Let $(y,z)$, $(z,y)\in P'$. Then there exist $q_1,\ldots,q_s$ and $q_1',\ldots,q_t'$ with 
\[
y\leq_Qq_1\leq_Q\cdots\leq_Q q_s\leq_Qz 
\hspace{1em}\mbox{and}\hspace{1em}
z\leq_Qq_1'\leq_Q\cdots\leq_Q q_t'=y.
\]
This gives $Q$-chains $y\leq_Q\cdots\leq_Qz\leq_Q\cdots\leq_Qy$ and $z\leq_Q\cdots\leq_Qy\leq_Q\cdots\leq_Qz$. 

\medskip

If none of these chains contains a pair of two consecutive elements equal to $x$ and $x'$, we can replace all $\leq_Q$-relations by $\leq_P$-relations and conclude $y=z$. Otherwise, such a pair, say in $y\leq_Q\cdots\leq_Qy$, yields the chains $x'\leq_P\cdots\leq_Py$ and $y\leq_P\cdots\leq_Px$ whence $x'\leq_P x$, contradiction.
\end{proof}

To investigate $\pF(\cl)$ further, we introduce the map
\[
\delta_\cl\colon\cl\times\cl\to
\Imm\cup\{\CO_X\}\cup-\Imm,\hspace{0.8em}\delta_\cl(\cA,\cB)=\cB-\cA.
\]
This is well-defined since $\cl$ is exceptional. Namely, for all distinct elements $A$, $B\in\cl$ there exists an exceptional order $\tO$ with either $A<_{\tO}B$ or $B<_{\tO}A$, that is, $\delta_\cl(A,B)\in-\Imm\cup\Imm$. By definition, we immediately obtain the

\begin{lemma}
\label{lem:deltaF}
Let $\cl$ be an exceptional set. Then
\[
\pF(\cl)=\delta_\cl^{-1}\big(\{\CO_X\}\cup-\Imm(X)\setminus(\Imm(X)\cap-\Imm(X))\big).
\]
\end{lemma}

\subsection{Effective exceptional sets}
\label{subsec:Div}
Let $\Eff(\cl)\subseteq\pF(\cl)$ be the subset of pairs $(A,B)$ with 
\[
B-A\in\Eff(X):=\{\CL\in\Pic(X)\mid\mr H^0(X,\CL)\not=0\};
\]
we also write $A\leq_{\mathrm{eff}}B$. Put differently, 
\[
\Eff(\cl)=\delta^{-1}_\cl\big(\{\CO_X\}\cup(\Eff(X)\cap-\Imm(X))\big)
\]

\begin{definition}
We call an exceptional set $\cl$ {\em \effec} if $\Eff(\cl)=\pF(\cl)$.
\end{definition}

Recall from Section~\ref{sec:ExObMut} that an exceptional sequence $(\CL_1,\ldots,\CL_N)$ was called strongly exceptional if for all $1\leq i,j\leq N$, $\CL_i^{-1}\otimes\CL_j$ lies in the {\em acyclic locus} 
\[
\nAcyc(X):=\{\CL\in\Pic X\mid\gH^k(X,\CL)=0
\mbox{ for all }k\geq 1\}.
\]
In particular, strong exceptionality is a condition on the underlying exceptional set. Via the generating set of its associated poset we can characterise this as follows.

\begin{proposition}
\label{prop:StrongExcep}
An exceptional set $\cl$ is strongly exceptional if and only if 
\[
\delta_\cl\big(\pF(\cl)\big)\subseteq\nAcyc(X).
\]
\end{proposition}

\begin{proof}
Since $\RGamma(X,\CO_X)=\kk$ only the converse needs proof; it is enough to consider $A\not=B\in\cl$. We assume without loss of generality that $A<_\tO B$ for some exceptional order $\tO$. In particular, $A-B\in\Imm(X)\subseteq\nAcyc(X)$. If $B-A\in\Imm(X)$, we are done. Otherwise, $(A,B)\in\pF$ so that by assumption, $\delta_\cl(A,B)=B-A\in\nAcyc(X)$.
\end{proof}

\begin{corollary}
\label{cor:ExSecPoset}
If an exceptional set $\cl$ is strongly exceptional, then
\[
\pF(\cl)=\Eff(\cl), 
\]
that is, a strongly exceptional sequence is \effec. 

\medskip

If in addition $\Eff(X)\cap-\Imm(X)\subseteq\nAcyc(X)$, then the converse holds, too.
\end{corollary}

\begin{proof}
Since $\nAcyc(X)\subseteq\Imm(X)\cup\Eff(X)$, Proposition~\ref{prop:StrongExcep} and Lemma~\ref{lem:deltaF} show that
\[
\delta_\cl\big(\pF(\cl)\big)\subseteq\{\mc O_X\}\cup\big(\Eff(X)\cap-\Imm(X)\big).
\]
In particular, $\pF(\cl)\subseteq\Eff(\cl)$. 

\medskip

Conversely, $\pF(\cl)=\Eff(\cl)$ implies
\[
\delta\big(\pF(\cl)\big)\subseteq\{\mc O_X\}\cup\big(\Eff(X)\cap-\Imm(X)\big).
\]
If $\Eff(X)\cap-\Imm(X)\subseteq\nAcyc(X)$, $\cl$ is strongly exceptional by Proposition~\ref{prop:StrongExcep}.
\end{proof}

\begin{remark}
As $\Eff(X)$ is a semigroup,
$\Eff(\cl)$ is transitive so that $\pO(\cl)=\Eff(\cl)$ for any effective sequence. Indeed, let $(A,B)$, $(B,C)\in\Eff(\cl)$ with $A\not=B$, $B\not=C$. Then $A<_\tO B$ and $B<_{\tO} C$ for any exceptional order whence $C-A\in-\Imm(X)$. Moreover, $B-A$, $C-B\in\Eff(X)$ so that $C-A\in\Eff(X)$. Hence $(A,C)\in\Eff(\cl)$.
\end{remark}

\section{Projective bundles over $\P^\lH$}
\label{sec:ProjBund}
We reexamine the setup of Example~\ref{exam:OrlovType} for $Y=\P^\lH$ and consider a vector bundle $\CE\to Y$ of $\rk(\CE) = \lV+1$. This induces the projective bundle
\[
\pi\colon X=\P(\CE)\to\P^\lH,
\]
together with the relative hyperplane section class $H$. We recall our convention 
\[
\P(\CE):=\Proj(\SymE)
\]
so that in particular, $\P(\CE)_y=\P(\CE_y)$ for the fibres over $y\in Y$.\footnote{Grothendieck's convention with $\Sym^\bullet\CE$ instead of $\SymE$ is also frequently used in the literature, but leads to different signs in the subsequent formul\ae.}

\medskip

Let $h$ be the pullback of a hyperplane section class of $\P^\lH$. Then
\[
\Pic(X) \cong \Z\cdot h\,\oplus\,\Z\cdot H;  
\]
in particular, the line bundle corresponding to $ih+jH$ is $\pi^*\CO_{\P^\lH}(i)\otimes\CO(jH)=\CO_X(ih+jH)$ which we simply denote $(i,j)$.
We refer to these coordinates on $\Pic(X)$ as {\em natural}. To lighten notation, we simply write $\CO(i)$ for $\CO_{\P^\lH}$ if there is no risk of confusion.

\begin{remark}
\label{rem:CLB}
For later use, we note that in natural coordinates the canonical sheaf of $\P(\CE)$ is given by 
\[
K_{\P(\CE)}=(-\lH-1-e)h-(\lV+1)H,
\]
where $\CO_{\P^\lH}(e)=\det(\CE)$. This follows from dualising and taking the determinant of the short exact sequences
\[
0\to\CT_{\P(\CE)|X}\to\CT_{\P(\CE)}\to\pi^*\CT_{\P^\lH}\to0.
\]
and
\[
0\to\CO_{\P(\CE)}\to\pi^*\CE(1)\to\CT_{\P(\CE)|X}\to0,
\]
cf.\ \cite[B.5.8]{fulton}.
\end{remark}

\begin{remark}
\label{rem:HyperSection}
Note that the projectivisation $\pi\colon\P(\CE)\to Y$ only determines $\CE$ up to a twist with a line bundle. Twisting $\CE$ with $\CO(D)$ yields the new hyperplane section class $H'=H-\pi^*D$~\cite[B.5.5]{fulton}.
\end{remark}

The well-known cohomology of line bundles on projective spaces and standard results of the cohomology of projective bundles yields the following

\begin{proposition}
\label{prop:CohomPBundleP}
Let $\CE\to\P^\lH$ be a vector bundle of rank $\lV+1$. Then
\[
\pi_*\CO_X(ih+jH)=\CO_{\P^\lH}(i)\otimes\Sym^j\CE^\vee
\]
for $j\geq0$. Furthermore, the line bundles $\CO_X(ih+jH)$ in the range $-\lV\leq j\leq -1$, subsequently referred to as the {\em horizontal strip} $\CH$, are immaculate, that is, $\CO_X(ih+jH)\in\Imm(X)$.
\end{proposition}

See also Figure~\ref{fig:immaculatePBundleP} for a schematic sketch.

\begin{figure}[ht]
\begin{tikzpicture}[scale=0.3]
\draw[color=gray!40] (-8.3,-7.3) grid (5.3,1.3);
\draw[->,color=black] (-8.3,-1) -- (6.3,-1) node[right]{$\scriptstyle h$};
\draw[->,color=black] (0,-7.3) -- (0,2.3) node[above]{$\scriptstyle H$};
\draw[very thick, color=magenta]
    (-8.3,-2) -- (5.3,-2);
\draw[very thick, color=magenta]
    (-8.3,-4) -- (5.3,-4);
\foreach \x in {-4,-3,-2,-1} {
 \fill[thick, color=magenta] (\x,-1) circle (7pt);
 \fill[thick, color=magenta] (\x+2,-5) circle (7pt);
};
\fill[pattern=north west lines, pattern color=magenta!40] (-8.3,-4) rectangle (5.3,-2);
\draw[color=magenta] (-2,-3) node[fill=white]{$\scriptstyle\CH$};
\fill[thick, color=black] (0,-1) circle (7pt) node[above right]{$\scriptstyle \CO$};
\end{tikzpicture}
\caption{Line bundles belonging to the immaculate locus of $\P(\CE^{\lV+1})\to\P^\lH$ are drawn in magenta (here for $\lH=4$ and $\lV=3$).}
\label{fig:immaculatePBundleP}
\end{figure}

\begin{remark}
\label{rem:Orlov}
A {\em horizontal chain} is a set of $\lH+1$ consecutive points $\CO_X(ih+jH)$, $i=i_0,\ldots,i_0+\lH$ on the line $[y=j]$, that is, it is the pull-back to $X$ of the Beilinson sequence $\CO_{\P^\lH}(i_0),\ldots,\CO_{\P^\lH}(i_0+\lH)$ on $\P^\lH$ twisted with $\CO_X(jH)$. The sets underlying exceptional sequences of Orlov-type consist thus of $\lV+1$ consecutive horizontal chains. In the language of~\cite{fesfano}, these are just the {\em trivial} sequences.
\end{remark}

\begin{remark}
\label{rem:EffNef}
$\R_{\geq0} \cdot h$ is an extremal ray of the pseudo-effective cone $\overline{\Eff(X)}$, where $\Eff(X)$ denotes the {\em effective cone} of (classes of) line bundles $\CL$ in $\Pic(X)$ with $H^0(X,\CL)\not=0$. If we assume $H$ to be effective, then the second extremal ray is of the form
\[
H_{\mathrm{eff}}=a_\eff\cdot h+b_\eff\cdot H
\]
with $a_\eff\leq 0$ and $b_\eff>0$. Since $\CO_{\P^\lH}(1)$ is very ample, $h$ is still nef on $X$. So $\R_{\geq0} \cdot h$ is also an extremal ray of the nef cone, and $\Nef(X)=\langle h,H_{\mathrm{nef}}\rangle\subseteq\Eff(X)$ for
\[
H_{\mathrm{nef}} = a_\nef\cdot h + b_\nef\cdot H 
\]
with $a_{\mathrm{eff}}/b_{\mathrm{eff}}\leq a_\nef/b_\nef$.
\end{remark}

In the following, we will refer to the coordinates on the Picard group induced by the primitive generators of the nef cone $h$ and $H_\nef$ as {\em nef coordinates}.

\medskip

The subsequent Sections~\ref{sec:toric} and~\ref{sec:CoTang} will consider two concrete instances of this setup, namely for a direct sum of line bundles and the cotangent bundle of $\P^\lH$.

\section{The projective bundle of a direct sum of line bundles}
\label{sec:toric}
The simplest case of a vector bundle $\CE\to\P^\lH$ is given by the direct sum of line bundles
\[
\CE=\CO_{\P^{\lH}}(c^0)\oplus\cdots\oplus\CO_{\P^{\lH}}(c^\lV)
\]
for integers $c^0\geq\cdots\geq c^\lV$. Since we are free to twist $\CE$ with a further line bundle we assume without loss of generality that $c^0=0$. We write $c=(c^1,\ldots,c^\lV)$ and 
\[
X=X(\lH,\lV;c):=\P(\CO_{\P^{\lH}}\oplus\CO_{\P^{\lH}}(c^1)\oplus\cdots\oplus\CO_{\P^{\lH}}(c^{\lV+1})).  
\]
In the so-called {\em product case} where $c=0$ we simply obtain $X=\P(\CO_{\P^\lH}^{\lV+1})=\P^\lV\times\P^\lH$. Otherwise, we speak of the {\em twisted case} which still yields a toric variety. In fact, Kleinschmidt's result~\cite{kle88} asserts that any complete, smooth toric variety of Picard rank $2$ arises this way. We therefore embrace the product and the twisted case by referring to the {\em toric case}.

\subsection{The immaculate locus}
The integer quantities 
\[
\xa:=-c^\lV\hspace{0.8em}\mbox{and}\hspace{0.8em}\xb:=-\sum_{i=1}^\lV c^i
\] 
associated with $X(\lH,\lV;c)$ are particularly important~\cite{fesfano}. Almost by definition we have the {\em effective inequality}
\begin{equation}
\label{eq:EffIneq}
\xb\leq\xa\lV.
\end{equation}

\medskip

In terms of the geometry of $X$ we find $\xb=-c_1(\mc E)$. The interpretation of $\xa$ is slightly more subtle. First,
\[
\pi_*\CO_X(-\xa h+H)=\CE^\vee(-\xa) = \CO_{\P^{\lH}}(-\xa)\oplus\cdots\oplus\CO_{\P^{\lH}}(-\xa-c^{\lV-1})\oplus\CO_{\P^{\lH}}
\]
so that the divisor $-\xa h+H$ is effective. More generally, for $ih+jH$ with $j>0$ we obtain
\begin{equation}
\label{eq:pushforwardSplit}
\pi_*\CO_X(ih+jH)=\CO_{\P^{\lH}}(i)\otimes\Sym^j\CE^\vee=\bigoplus_{1\leq k_1,\ldots,k_j\leq\lV}\CO_{\P^{\lH}}(-c^{k_1}-\cdots-c^{k_j}+i).
\end{equation}
If $i<0$ with $i+\alpha j<0$, then $\alpha j\leq-c^{k_1}-\cdots-c^{k_j}$ for all $1\leq k_1,\ldots,k_j\leq\lV$, that is, $ih+jH$ is not effective. 
In particular, $-\xa h+H$ is an extremal ray of the pseudo-effective 
cone of $X$, or in the notation of Remark~\ref{rem:EffNef}~(ii), $a_\eff=-\xa$ and $b_\eff=1$. Since $X$ is toric, the pseudo-effective cone equals the effective cone whence
\[
\Eff(X)=\langle h,\,-\xa h+H\rangle.
\]

\medskip

For the nef cone we find
\begin{equation}
\label{eq:NefCone}
\Nef(X)=\langle h,\,H\rangle.
\end{equation}
Indeed, assume to the contrary that $a'h+b'H$ with $a'<0$ and $b'>0$ is an 
extremal ray of the nef cone. Taking a suitable positive multiple of this divisor if necessary we may assume that $a'<-\lH$ and $(a'+1)h+b'H$ is very ample. From~\eqref{eq:pushforwardSplit} it follows that $\CO_{\P^{\lH}}(a)$ is a direct summand of $\pi_* \CO_X(ah+bH)$ by choosing all $i_j=0$. Since $a<-\lH$, $\CO_X(ah)$ has non-trivial $\lH$-th cohomology and so has $\CO((a+1)h+bH)$, contradicting the very ampleness. Consequently, $H$ is an extremal ray of the nef cone. In particular, the natural coordinates on $\Pic(X)$ are nef, cf.\ Remark~\ref{rem:EffNef}.

\medskip

Next we describe the immaculate locus which is the complement in $\Pic(X)$ of the (not necessarily disjoint) union of the sets
\[
\CH^k(X):=\{\CL\mid H^k(X,\CL)\not=0\},\hspace{0.8em}k=0,\ldots,\lH+\lV.
\]
In~\cite[$(5.1)$]{immaculate} it was shown that that the these sets are either empty or the strongly convex cones given by
\begin{equation}
\label{eq:MacLocusTC}
\renewcommand{\arraystretch}{1.3}
\begin{array}{lcl}
\CH^0(X)&=&\langle(1,0),(-\xa,1)\rangle
\subseteq \Z\times\Z_{\geq 0}
\\
\CH^\lH(X)&=&(-\lH-1,0)+ \langle (-1,0),\;(0,1)\rangle
\subseteq\Z\times\Z_{\geq 0}
\\
\CH^{\lV}(X)&=&(\xb,-\lV-1)+\langle (1,0),\;(0,-1)\rangle
\subseteq \Z\times\Z_{< 0}\\
\CH^{\lH+\lV}(X)&=&(\xb-\lH-1,-\lV-1)+ \langle (-1,0),\;(\xa,-1)\rangle\subseteq \Z\times\Z_{<0}.
\end{array}
\end{equation}
The immaculate locus $\Imm(X)$ is thus completely determined by the quantities $\xa$ and $\xb$. The canonical bundle is given by $K_X=(\xb-\lH-1,-\lV-1)$. Figure~\ref{fig:ImmLoc} illustrates the typical shape for the product and twisted case.

\begin{figure}[ht]
\newcommand{\spaceA}{\hspace{10pt}}
\newcommand{\scaleA}{0.3}
\begin{tikzpicture}[scale=\scaleA]
\draw[color=gray!40] (-8.3,-11.3) grid (3.3,5.3);
\draw[->,color=black] (0,-1) -- (4.3,-1) node[right]{$\scriptstyle h$};
\draw[->,color=black] (0,-1) -- (0,6.3) node[above]{$\scriptstyle H$};
\draw[thick, color=magenta]
    (-8.3,-2) -- (-4,-2) (-1,-2) -- (3.3,-2);
\draw[thick, color=magenta]
    (-8.3,-4) -- (-4,-4) (-1,-4) -- (3.3,-4);
\draw[thick, color=magenta]
    (-4,5.3) -- (-4,-2) (-4,-4) -- (-4,-11.3);
\draw[thick, color=magenta]
    (-1,5.3) -- (-1,-2) (-1,-4) -- (-1,-11.3);
\fill[pattern=north west lines, pattern color=magenta!40] (-8.3,-4) rectangle (3.3,-2);
\fill[pattern=north west lines, pattern color=magenta!40] (-4,5.3) rectangle (-1,-11.3);
\fill[pattern=north east lines, pattern color=yellow!40] (0,-1) rectangle (3.3,5.3);
\fill[pattern=north west lines, pattern color=oliwkowy!60] (0,-1) rectangle (3.3,5.3);
\draw[thick, color=oliwkowy!60] (0,-1) -- (3.3,-1) (0,-1) -- (0,5.3);
\fill[thick, color=black] (0,-1) circle (7pt) node[below]{$\scriptstyle \CO_X$};
\draw (3,0.1) node[fill=white] {$\scriptstyle\CH^0\neq0$};
\draw (5,2) node[fill=white][above] {$\scriptstyle \Eff(X)=\Nef(X)$};
\fill[pattern=north east lines, pattern color=oliwkowy!60] (-8.3,-1) -- (-5,-1) -- (-5,5.3) -- (-8.3,5.3);
\draw[thick, color=oliwkowy!60] (-8.3,-1) -- (-5,-1) -- (-5,5.3);
\draw (-7,0.1) node[fill=white] {$\scriptstyle \CH^\lH\neq0$};
\fill[pattern=north west lines, pattern color=oliwkowy!60] (0,-5) rectangle (3.3,-11.3);
\draw[very thick, color=oliwkowy!60] (0,-5) -- (3.3,-5) (0,-5) -- (0,-11.3);
\draw (3,-6) node[fill=white] {$\scriptstyle \CH^\lV\neq0$};
\fill[pattern=north east lines, pattern color=oliwkowy!60] (-8.3,-5) rectangle (-5,-11.3);
\draw[thick, color=oliwkowy!60] (-8.3,-5) -- (-5,-5) -- (-5,-11.3);
\fill[thick, color=black] (-5,-5) circle (7pt) node[above right]{$\scriptstyle K_X$};
\draw (-7.5,-6.1) node[fill=white] {$\scriptstyle\CH^{\lH+\lV}\neq0$};
\end{tikzpicture}
\spaceA
\begin{tikzpicture}[scale=\scaleA]
\draw[color=gray!40] (-8.3,-11.3) grid (4.3,5.3);
\draw[->,color=black] (0,-1) -- (5.3,-1) node[right]{$\scriptstyle h$};
\draw[->,color=black] (0,-1) -- (0,6.3) node[above]{$\scriptstyle H$};
\draw[thick, color=magenta]
    (-8.3,-2) -- (4.3,-2);
\draw[thick, color=magenta]
    (-8.3,-4) -- (4.3,-4);
\foreach \y in {0,...,3} {
    \foreach \x in {-3,...,-\y} {
        \fill[thick, color=magenta] (\x-1,\y-1) circle (7pt);
        \fill[thick, color=magenta] (-\x-2,-\y-5) circle (7pt);
    }
}
\fill[pattern=north west lines, pattern color=magenta!40] (-8.3,-4) rectangle (4.3,-2);
\fill[pattern=north east lines, pattern color=yellow!40] (4.3,-1) -- (0,-1)-- (0,5.3) -- (4.3,5.3);
\fill[pattern=north west lines, pattern color=oliwkowy!60] (0,-1)-- (-6.3,5.3) -- (4.3,5.3) -- (4.3,-1);
\draw[thick, color=oliwkowy!60] (0,-1) -- (4.3,-1) (0,-1) -- (-6.3,5.3);
\fill[thick, color=black] (0,-1) circle (7pt) node[below]{$\scriptstyle\CO_X$};
\draw (5.6,-0.8) node[fill=white][above] {$\scriptstyle\CH^0\neq0:\;\Eff(X)$};
\draw (4.2,3) node[fill=white][above] {$\scriptstyle\Nef(X)$};
\fill[pattern=north east lines, pattern color=oliwkowy!60] (-8.3,-1) -- (-5,-1) -- (-5,5.3) -- (-8.3,5.3);
\draw[thick, color=oliwkowy!60] (-8.3,-1) -- (-5,-1) (-5,5.3) -- (-5,-1);
\draw (-7,-0.8) node[fill=white][above] {$\scriptstyle\CH^\lH\neq0$};
\fill[pattern=north west lines, pattern color=oliwkowy!60] (2,-5) -- (4.3,-5) -- (4.3,-11.3) -- (2,-11.3);
\draw[thick, color=oliwkowy!60] (2,-5) -- (4.3,-5) (2,-5) -- (2,-11.3);
\draw (4,-5.2) node[fill=white][below] {$\scriptstyle\CH^\lV\neq0$};
\fill[pattern=north east lines, pattern color=oliwkowy!60] (-8.3,-5) -- (-3,-5) -- (3.3,-11.3) -- (-8.3,-11.3);
\draw[thick, color=oliwkowy!60] (-8.3,-5) -- (-3,-5) -- (3.3,-11.3);
\fill[thick, color=black] (-3,-5) circle (7pt) node[above]{$\scriptstyle K_X$};
\draw (-6.5,-5.2) node[fill=white][below] {$\scriptstyle\CH^{\lH+\lV} \neq0$};
\end{tikzpicture}
\caption{The immaculate locus for $X(\lH,\lV;c)$ with $\lH=4$ and $\lV=3$. On the left hand side the product case with $c=0$, on the right hand side the twisted case with $\xa=1$ and $\xb=2$. The nef cone is shaded in yellow.}
\label{fig:ImmLoc}
\end{figure}

\subsection{Maximal exceptional sequences in the toric case}
In view of establishing Theorem~\thmStrongExSec\ in the toric case, we recall the toric version of Theorem~\thmExSec\ from~\cite[Theorems 7.7 and 10.2]{fesfano}. The following arguments use nef coordinates.

\medskip

Let $\sigma$ be the $\Pic(X)$-involution given by
\begin{equation}
\label{eq:sigmaInv}
\sigma:\Z^2\to\Z^2,\quad(i,j)\mapsto(j,i).
\end{equation}
For an exceptional set $\cl$ we may assume without loss of generality that it fits into a rectangle of at most $2(\lH+1)$ rows. In the product case, this holds possibly after substituting $\cl$ by $\sigma(\cl)$, cf.~\cite[Theorems 6.1 and 9.8]{fesfano}. 

\medskip

Then up to normalisation, that is, up to twisting $\cl$ with a suitable line bundle, $\cl$ is one of the following sets. First, consider an {\em admissible set}
\[
\cX\subseteq\DeltaUp:=\{(i,j)\in\Z^2\mid -\xb<i,\;\lV<j,\; 
\mbox{ and }i+\xa j\leq\lH+\xa(\lV+1)-\xb\}
\]
determined by the properties below.
\begin{enumerate}
\item[{\rm (Ai)}] Vertical spatial boundedness: The horizontal layers 
$$
\cX(k):=\cX\cap[y=k]
$$ 
are empty for $k\not=\lV+1,\ldots,2\lV+1$. 
\item[{\rm (Aii)}] Normalisation: If $X\not=\varnothing$, then $(\lH-\xb,\lV+1)\in\cX$.
\item[{\rm (Aiii)}] Ascending slimness: For each $k\geq\lV+1$ and $(x,k+1)\in \cX(k+1)$ 
the points $(x,k),\,(x+1,k),\,\ldots,\,(x+\xa,k)$
belong to $\cX(k)$, that is, ascending layers of $\cX$ become smaller. In particular, $\cX(k)=\varnothing$ implies $\cX(k+1)=\varnothing$.
\end{enumerate}

\medskip

Second, we extend the layers $\cX(k)\not=\varnothing$ successively to the left until we obtain a horizontal chain in $[y=k]$ (cf.\ Remark~\ref{rem:Orlov}). Whenever $\cX(k)=\varnothing$, we freely choose a horizontal chain. Denote by $\cY(k)$ the new points in the layer $[y=k]$. Since $\cX(k)\subseteq\DeltaUp$, $\cY(k)$ is necessarily nonempty. We let $\cY:=\bigcup_k\cY(k)$ and define
\[
\cZ:=\cY+(\xb,-\lV-1).
\]
Then $\cZ\cup\cX$ defines a \meset\ on $X$, and any \meset\ arises this way. See also Figure~\ref{fig:MesetTV} which visualises this construction; elements of $\cX$, $\cY$, and $\cZ$ correspond to red, blue and green points, respectively.

\begin{figure}[ht]
\newcommand{\spaceA}{\hspace{10pt}}
\newcommand{\scaleA}{0.3}
\begin{tikzpicture}[scale=\scaleA]
\draw[color=gray!40] (-2.3,-1.3) grid (8.3,8.3);
\draw[very thick, color=pink]
    (0.5,8.5) -- (0.5,3.50) -- (4.5,3.50) -- (4.5,8.5)
    (2.7,8.1) node{$\scriptstyle\DeltaUp$};
\fill[thick, color=colorX]
    (2,4) circle (7pt) (3,4) circle (7pt)
    (4,4) circle (7pt) (2,5) circle (7pt)
    (3,5) circle (7pt) (4,5) circle (7pt)
    (3,6) circle (7pt);
\draw[thick, color=black]
    (0,0) circle (10pt) node[below left]{$\scriptstyle\CO_X$};
\draw[thick, color=black]
    (4,4) circle (10pt) node[right]{$\hspace{4pt}\scriptstyle(\lH,\lV+1)$};
\draw[thick, color=colorY]
    (1,4) circle (7pt) (0,5) circle (7pt)
    (1,5) circle (7pt) (-1,6) circle (7pt)
    (0,6) circle (7pt) (1,6) circle (7pt)
    (2,6) circle (7pt) (6,7) circle (7pt)
    (7,7) circle (7pt) (3,7) circle (7pt)
    (4,7) circle (7pt) (5,7) circle (7pt)
    (0,4) circle (7pt);
\fill[yshift=-4cm, thick, color=colorZ]
    (1,4) circle (7pt) (0,5) circle (7pt)
    (1,5) circle (7pt) (-1,6) circle (7pt)
    (0,6) circle (7pt) (1,6) circle (7pt)
    (2,6) circle (7pt) (6,7) circle (7pt)
    (7,7) circle (7pt) (3,7) circle (7pt)
    (4,7) circle (7pt) (5,7) circle (7pt)
    (0,4) circle (7pt);
\end{tikzpicture}
\spaceA
\begin{tikzpicture}[scale=\scaleA]
\draw[color=gray!40] (-5.3,-1.3) grid (8.3,8.3);
\draw[very thick, color=pink]
    (-1.5,8.0) -- (-1.5,3.43) -- (3,3.43) -- (-1.5,8)
    (0.2,8.1) node{$\scriptstyle\DeltaUp$};
\fill[thick, color=colorX]
    (0,4) circle (7pt) (1,4) circle (7pt)
    (2,4) circle (7pt) (0,5) circle (7pt)
    (0,6) circle (7pt) (1,5) circle (7pt);
\draw[thick, color=black]
    (0,0) circle (10pt) node[below left]{$\scriptstyle\CO_X$};
\draw[thick, color=black]
    (2,4) circle (10pt) node[right]{$\hspace{4pt}\scriptstyle(\lH-\xb,\lV+1)$};
\draw[thick, color=colorY]
    (-1,4) circle (7pt) (-2,4) circle (7pt)
    (-3,5) circle (7pt) (3,7) circle (7pt)
    (-2,5) circle (7pt) (-1,5) circle (7pt)
    (-4,6) circle (7pt) (-3,6) circle (7pt)
    (-2,6) circle (7pt) (-1,6) circle (7pt) 
    (-1,7) circle (7pt) (0,7) circle (7pt) 
    (1,7) circle (7pt) (2,7) circle (7pt);
\fill[thick, color=colorZ]
    (0,0) circle (7pt) (1,0) circle (7pt) 
    (-1,1) circle (7pt) (5,3) circle (7pt)
    (0,1) circle (7pt) (1,1) circle (7pt)
    (-2,2) circle (7pt) (-1,2) circle (7pt) 
    (0,2) circle (7pt) (1,2) circle (7pt)
    (1,3) circle (7pt) (2,3) circle (7pt) 
    (3,3) circle (7pt) (4,3) circle (7pt);
\end{tikzpicture}
\caption{The sets $\cX$, $\cY$ and $\cZ$ for \meset s $\cl$ on $X(4,3;0)$ (left-hand side) and $X(4,3;c)$ with $\xa=1$ and $\xb=2$ (right-hand side).}
\label{fig:MesetTV}
\end{figure}

\begin{remark}
An exceptional order on $\cZ\cup\cX$ is given by the {\em vertical lexicographical order} $<_{\mathrm{vert}}$ defined by
\begin{center}
$(i_1,j_1)<_{\mathrm{vert}}(i_2,j_2)$ if and only if 
$j_1<j_2$ or $\big(j_1=j_2$ and $i_1<i_2\big)$.
\end{center}
Similarly, we define the horizontal lexicographical order. In particular, any \meset\ can be vertically ordered (twisted case) or vertically / horizontally ordered (product case). By Proposition~\ref{prop:pFpO}, the maximal \effec\ sequences are precisely those \meset s $\cl$ which can be ordered both vertically and horizontally. In fact, $\pO(\cl)$ is contained in any global term order used in computer algebra.

\medskip

As a remarkable matter of fact, this feature is proper to \meset s and Picard rank $2$. There are counterexamples for non-maximal non-extendable exceptional sets~\cite[Example 8.1]{fesfano} and for \meset s on $\P^1\times\P^1\times\P^1$~\cite[Example 3.5]{p1p1p1}.
\end{remark}

Summarising, an exceptional set is completely determined by $\cZ$. For practical reasons we let
\[
\free:=\{k\in\Z\mid0\leq k\leq\lV\;\mbox{and }\cX(k+\lV+1)=\varnothing\}
\subseteq\{0,\ldots,\lV\},
\]
and
\[
\Zfree:=\bigcup_{k\in\free}\cl(k)\quad\mbox{and}\quad\Zbound:=\cZ\setminus\bigcup_{k\in\free}\cl(k)
\]
be the {\em free} and the {\em residual part} of $\cZ$, respectively. The reason for the subscript $-\infty$ will become clear in Subsection~\eqref{subsec:StrongExSecTV}.

\begin{remark}
\cite[Proposition 5.3]{fesfano} implies that a \meset\ $\cl$ on $X(\lH,\lV;c)$ which can be vertically ordered satisfies
\begin{equation}
j_2-j_1\geq-\lV\hspace{0.8em}\mbox{for all }(i_1,j_1)<(i_2,j_2)\mbox{ in }\cl.
\end{equation}
Therefore, if two consecutive line bundles $(\CL_0,\CL_1)$ are not vertically ordered, mutating to the left yields $(\leftMut_{\CL_0}\CL_1,\CL_0)=(\CL_1,\CL_0)$, cf.\ Remark~\ref{rem:Mut} and Figure~\ref{fig:CF}. In particular, we can achieve vertical lexicographical order by successively mutating line bundles to the left. This gives rise to a fullness preserving operator 
\[
\mr{Lex}:\{\mbox{\mes}\}\to\{\mbox{\mes}\}
\]
which sends a given \mes\ to a vertically ordered \mes. Combining this operator with helixing to the right yields the {\em helex} (!) {\em operator} $\helex:=\helixR\circ\mr{Lex}$ which also preserves fullness. Theorem~\thmExSec\ then implies that up to applying $\helex$ a finite number of times, any \meset\ can be transformed into a \mes\ of Orlov type, cf.\ also~\cite[Proposition 7.2]{fesfano}. Theorem~\thmFull\ follows immediately. 
\end{remark}

\subsection{Strongly exceptional maximal sets}
\label{subsec:StrongExSecTV}
Knowing the shape of the \meset s we can in principle determine the generating sets $\pF$ of the associated partial order. At any rate, we derive a direct criterion for strongness.

\medskip

Let
\begin{align*}
\mc F&:=-\Imm(X)\setminus\big(-\Imm(X)\cap\Imm(X)\big)\\
&\phantom{:}=-\Imm(X)\setminus\{(i,j)\in-\Imm(X)\mid0<|i|\leq\lH,0<|j|\leq\lV\},  
\end{align*}
cf.\ also Figure~\ref{fig:CF}. By Lemma~\ref{lem:deltaF},
\begin{equation}
\label{eq:CompF}
\pF(\cl)=\bigcup_{\mc L\in\cl}\{\mc L\}\times\big(\cl\cap(\mc L+\mc F)\big). 
\end{equation}
For strong exceptionality we need $\pF(\cl)\subseteq\nAcyc(X)$, the acyclic locus of $X$, cf.\ Proposition~\ref{prop:StrongExcep}. From Figure~\ref{fig:StrongExSecPC} (product case) and Figure~\ref{fig:BadDisplaced} (twisted case) we immediately gather that any pair $(\CL,\CL')\in\pF(\cl)\setminus\Eff(\cl)$ must involve the free part $\Zfree(\cl)$: If the horizontal chains are torn too far apart we obtain pairs $(\CL,\CL')$ in $\pF$ with $\CL'-\CL\not\in\nAcyc(X)$.

\begin{figure}
\newcommand{\spaceA}{\hspace{10pt}}
\newcommand{\scaleA}{0.3}
\begin{tikzpicture}[scale=\scaleA]
\draw[color=gray!40] (-6.3,-5.3) grid (8.3,7.3);
\fill[pattern=north west lines, pattern color=oliwkowy!60] (0,0) rectangle (8.3,7.3);
\draw[thick, color=oliwkowy!60] (0,7.3) -- (0,3) (0,1) -- (0,0) -- (1,0) (4,0) -- (8.3,0);
\draw (7,5) node[fill=white] {$\scriptstyle\Eff(X)$};
\fill[thick, color=black] (0,0) circle (7pt) node[below left]{$\scriptstyle\CO_X$};
\draw[thick, color=magenta]
    (-6.3,3) -- (-5,3)   (0,3) -- (1,3) (4,3) -- (8.3,3)
    (-6.3,1) -- (-5,1)   (0,1) -- (1,1) (4,1) -- (8.3,1)
    (-5,3) -- (-5,1) (0,3) -- (0,1);
\fill[pattern=north west lines, pattern color=magenta!40] 
    (-6.3,1) rectangle (-5,3) (0,1) rectangle (8.3,3);
\draw[thick, color=magenta]
    (1,-5.3) -- (1,-4)   (1,0) -- (1,1) (1,3) -- (1,7.3)
    (4,-5.3) -- (4,-4)   (4,0) -- (4,1) (4,3) -- (4,7.3)
    (1,0) -- (4,0) (1,-4) -- (4,-4);
\fill[pattern=north west lines, pattern color=magenta!40] 
    (1,0) rectangle (4,7.3) (1,-4) rectangle (4,-5.3);
\foreach \x in {-1,...,-4} 
    {
    \foreach \y in {1,2,3}
        {
        \draw[thick, color=black] (\x,\y) circle (7pt);
        }
    }
\foreach \x in {1,...,4} 
    {
    \foreach \y in {1,2,3}
        {
        \draw[thick, color=black] (\x,-4+\y) circle (7pt);
        }
    }
\end{tikzpicture}
\spaceA
\begin{tikzpicture}[scale=\scaleA]
\draw[color=gray!40] (-6.3,-5.3) grid (8.3,7.3);
\fill[pattern=north west lines, pattern color=oliwkowy!60] (0,0) -- (-6.3,6.3) -- (-6.3,7.3) -- (8.3,7.3) -- (8.3,0);
\draw[thick, color=oliwkowy!60] (-6.3,6.3) -- (-3,3) (-1,1) -- (0,0) -- (8.3,0);
\draw (7,5) node[fill=white] {$\scriptstyle\Eff(X)$};
\draw[thick, color=magenta]
    (-6.3,3) -- (-5,3)  (-3,3) -- (8.3,3)
    (-6.3,1) -- (-5,1)  (-1,1) -- (8.3,1);
\fill[pattern=north west lines, pattern color=magenta!40] 
    (-6.3,1) rectangle (-5,3) (-1,1) -- (-3,3) -- (8.3,3) -- (8.3,1) -- (-1,1);
\draw[thick, color=magenta]
    (-3,3) -- (-1,1)  (-5,3) -- (-5,1); 
\fill[color=magenta] (1,0) circle (7pt) (2,0) circle (7pt) (3,0) circle (7pt) (4,0) circle (7pt);
\fill[color=magenta] (-1,4) circle (7pt) (0,4) circle (7pt) (1,4) circle (7pt) (2,4) circle (7pt) (-1,5) circle (7pt) (0,5) circle (7pt) (1,5) circle (7pt) (-1,6) circle (7pt) (0,6) circle (7pt) (-1,7) circle (7pt);
\draw[thick, color=black] 
    (-2,1) circle (7pt) (-3,1) circle (7pt) (-4,1) circle (7pt)
    (-3,2) circle (7pt) (-4,2) circle (7pt) (-4,3) circle (7pt);
\draw[thick, color=black] 
    (2,-1) circle (7pt) (3,-1) circle (7pt) (4,-1) circle (7pt)
    (3,-2) circle (7pt) (4,-2) circle (7pt) (4,-3) circle (7pt);
\fill[thick, color=black] (0,0) circle (7pt) node[below left]{$\scriptstyle\CO_X$};
\end{tikzpicture}
\caption{The magenta shaded area together with the boundary forms $\CF$ (here for $\lH=4$ and $\lV=3$ and $\xa=1$ and $\xb=2$). The black circles represent the intersection $-\Imm(X)\cap\Imm(X)$.
}
\label{fig:CF}
\end{figure}

\begin{figure}[ht]
\newcommand{\spaceA}{\hspace{10pt}}
\newcommand{\scaleA}{0.3}
\begin{tikzpicture}[scale=\scaleA]
\draw[color=gray!40] (-3.3,-5.3) grid (8.3,6.3);
\draw[very thick, color=pink]
    (0.5,6.5) -- (0.5,1.5) -- (4.5,1.5) -- (4.5,6.5)
    (2.7,5.1) node{$\scriptstyle\DeltaUp$};
\fill[yshift=-2cm, thick, color=colorX]
    (2,4) circle (7pt) (3,4) circle (7pt)
    (4,4) circle (7pt) (2,5) circle (7pt)
    (3,5) circle (7pt) (4,5) circle (7pt)
    (3,6) circle (7pt);
\draw[yshift=-2cm, thick, color=colorX]
    (2,4) circle (10pt) 
    (6,4) circle (10pt);
\draw[thick, color=black]
    (-5,2) node{$\scriptstyle\cl(\infty)^L$}
    (-5,1) node{$\scriptstyle\cl(5)^L$}
    (-5,0) node{$\scriptstyle\cl(4)^L$}
    (-5,-1) node{$\scriptstyle\cl(3)^L$}
    (-5,-2) node{$\scriptstyle\cl(-\infty)^L$};
\draw[thick, color=colorZ]
    (-2,-2) circle (10pt) (2,-2) circle (10pt)
    (0,-1) circle (10pt) (4,-1) circle (10pt)
    (1,0) circle (10pt) (5,0) circle (10pt)
    (-2,1) circle (10pt) (2,1) circle (10pt)
    ;
\fill[yshift=-4cm, thick, color=colorZ]
    (1,0) circle (7pt) (0,1) circle (7pt)
    (1,1) circle (7pt) (0,2) circle (7pt) 
    (1,2) circle (7pt) (0,0) circle (7pt)
    (2,2) circle (7pt) (0,3) circle (7pt)
    (1,3) circle (7pt) (2,3) circle (7pt)
    (3,3) circle (7pt) (4,3) circle (7pt)
    (1,4) circle (7pt) (2,4) circle (7pt) 
    (3,4) circle (7pt) (4,4) circle (7pt) 
    (5,4) circle (7pt) (0,5) circle (7pt) 
    (1,5) circle (7pt) (2,5) circle (7pt) 
    (-1,5) circle (7pt) (-2,5) circle (7pt);
\draw[thick, color=black]
    (0,-4) circle (10pt) node[below]{$\scriptstyle\CO_X$};
\end{tikzpicture}
\spaceA
\begin{tikzpicture}[scale=\scaleA]
\draw[color=gray!40] (-3.3,-5.3) grid (8.3,6.3);
\draw[very thick, color=pink]
    (0.5,6.5) -- (0.5,1.5) -- (4.5,1.5) -- (4.5,6.5)
    (2.7,5.1) node{$\scriptstyle\DeltaUp$};
\fill[yshift=-2cm, thick, color=colorX]
    (2,4) circle (7pt) (3,4) circle (7pt)
    (4,4) circle (7pt) (2,5) circle (7pt)
    (3,5) circle (7pt) (4,5) circle (7pt)
    (3,6) circle (7pt);
\draw[yshift=-2cm, thick, color=colorX]
    (2,4) circle (10pt) 
    (6,4) circle (10pt);
\draw[thick, color=black]
    (10,2) node{$\scriptstyle\cl(\infty)^R$}
    (10,1) node{$\scriptstyle\cl(5)^R$}
    (10,0) node{$\scriptstyle\cl(4)^R$}
    (10,-1) node{$\scriptstyle\cl(3)^R$}
    (10,-2) node{$\scriptstyle\cl(-\infty)^R$};
\draw[yshift=-4cm, thick, color=colorZ]
    (-2,2) circle (10pt) (2,2) circle (10pt)
    (-1,3) circle (10pt) (3,3) circle (10pt)
    (1,4) circle (10pt) (5,4) circle (10pt)
    (1,5) circle (10pt) (5,5) circle (10pt);
\fill[yshift=-4cm, thick, color=colorZ]
    (1,0) circle (7pt) (0,1) circle (7pt)
    (1,1) circle (7pt) (-1,2) circle (7pt)
    (0,2) circle (7pt) (1,2) circle (7pt)
    (2,2) circle (7pt) (0,3) circle (7pt)
    (1,3) circle (7pt) (2,3) circle (7pt)
    (3,3) circle (7pt) (-1,3) circle (7pt)
    (1,4) circle (7pt) (2,4) circle (7pt) 
    (3,4) circle (7pt) (4,4) circle (7pt) 
    (5,4) circle (7pt) (5,5) circle (7pt) 
    (1,5) circle (7pt) (2,5) circle (7pt) 
    (3,5) circle (7pt) (4,5) circle (7pt)
    (0,0) circle (7pt);
\draw[thick, color=black]
    (0,-4) circle (10pt) node[below]{$\scriptstyle\CO_X$};
\end{tikzpicture} 
\caption{Two \meset s on $X=X(4,7;0)$. On the left-hand side, the layer $\cl(5)$ is displaced: $(\cl(4)^R,\cl(5)^L)\in\pF(\cl)$, but $\cl(5)^L-\cl(4)^R\not\in\nAcyc(X)$. 
The \meset\ on the right-hand side is strongly exceptional.}
\label{fig:StrongExSecPC}
\end{figure}

\medskip

To formulate this more neatly, we introduce the ``thick layers''
\[
\textstyle
\cl(-\infty):=\Zbound
\hspace{0.7em}\mbox{and}\hspace{0.7em}
\cl(\infty):=\cX
\]
and extend $\free$ to
\[
\layers:=\free\cup\{-\infty,\infty\}\subseteq\{-\infty,0,\ldots,\lV,\infty\}.
\]
Since any exceptional order on $\cl$ restricts to the natural total order of the integers on $\cl(k)\subseteq[y=k]\cong\Z$ (cf.~\cite[Lemma 5.1]{fesfano}), we can define the points
\[
\cl(k)^L:=\min\cl(k)
\hspace{0.9em}\mbox{and}\hspace{0.9em}
\cl(k)^R:=\max\cl(k),\quad k=0,\ldots,\lV.
\]
Provided that $\Zbound$ (and thus $\cX$) is not empty, we also set
\begin{equation}
\label{eq:XLZR}
\cl(\infty)^L:=\min\cl(\lV+1)
\hspace{0.9em}\mbox{and}\hspace{0.9em}
\cl(-\infty)^R:=\max\cl(k_0)
\end{equation}
where $k_0$ is the maximal integer such that $\Zbound\cap[y=k_0]\not=\varnothing$. Put differently, $\cl(\infty)^L\in\cX$ is the lower left corner of $\cX$ while $\cl(-\infty)^R$ is the upper right corner of $\Zbound$. In particular, vertical slimness implies
\begin{equation}
\label{eq:ZRandXL}
A\leq\cl(-\infty)^R\mbox{ and }\cl(\infty)^L\leq B\hspace{0.8em}\mbox{for all }A\in\Zbound,\,B\in\cX.
\end{equation}
Finally, we let
\[
\cl(\infty)^R:=\cl(\infty)^L+(\lH,0)\quad\mbox{and}\quad\cl(-\infty)^L:=\cl(-\infty)^R-(\lH,0)
\]
so that
\[
\cl(k)^R - \cl(k)^L = (\lH,0)
\]
for all $k\in\layers$. Note, however, that $\cl(\infty)^R\notin\cX$ and $\cl(-\infty)^L\notin\Zbound$. 

\medskip

We call a layer $\cl(k)$, $k\in\layers$, {\em displaced} if there exists $k'\in\layers$ with 
\[
k'<k\hspace{0.8em}\mbox{and}\hspace{0.8em}\cl(k')^L\not\leq_{\mathrm{nef}}\cl(k)^L,
\]
cf.~again Figure~\ref{fig:StrongExSecPC} for illustration.

\medskip

Then the discussion at the beginning of Subsection~\eqref{subsec:StrongExSecTV} implies

\begin{theorem}[Theorem~\thmStrongExSec\ for the toric case]
\label{thm:pO}
Let $X=X(\lH,\lV;c)$. A \meset\ $\cl$ is strongly exceptional if and only if there are no displaced layers. Equivalently, $\leq_\nef$ induces a total order on the set $\{\cl(k)^L\mid k\in\layers\}$. 
\end{theorem}

\subsection{Proper \effec\ \meset s.}
\label{subsec:Effective_TV}
Finally, we look for \effec\ \meset s which are {\em proper}, that is, which are not strongly exceptional. By Corollary~\ref{cor:ExSecPoset}, their existence requires $\Eff(X)\cap-\Imm(X)\not\subseteq\nAcyc(X)$.

\begin{proposition}
\label{prop:AcyclCrit}
$\Eff(X)\cap-\Imm(X)\subseteq\nAcyc(X)$ if and only if $\lH\geq\alpha\lV$. In particular, $X$ must be necessarily Fano.
\end{proposition}

\begin{proof}
Assume that $\CL=(i,j)\in\Eff(X)\cap-\Imm(X)$ is not acyclic. This happens precisely if there exists $1\leq k\leq\dim X=\lH+\lV$ with $H^k(X,\CL)\not=0$. Since $\CL\in\Eff(X)$ implies $j\geq0$, we deduce from~\eqref{eq:MacLocusTC} on Page~\pageref{eq:MacLocusTC} that $k=\lH$. Hence   
\[
\Eff(X)\cap-\Imm(X)\subseteq\nAcyc(X)\hspace{0.8em}\mbox{if and only if}\hspace{0.8em}\Eff(X)\cap-\Imm(X)\cap\CH^\lH(X)=\varnothing.
\]

\medskip

Since $\xb\leq\xa\lV$ by the effective inequality~\eqref{eq:EffIneq} on Page~\pageref{eq:EffIneq}, points of minimal abscissa in $\Eff(X)\cap-\Imm(X)$ must be in the horizontal strip and therefore in $\{x=-\xa\lV\}$, cf.\ also Figure~\ref{fig:AcyclCrit}. Hence $\Eff(X)\cap-\Imm(X)\cap\CH^\lH(X)$ is empty if and only if $-\lH-1<-\xa\lV$, or equivalently, $\lH\geq\xa\lV$. In particular, $\lH\geq\xb$, which by~\cite{kle88} implies that $X$ must be Fano.
\end{proof}

\begin{figure}
\newcommand{\spaceA}{\hspace{10pt}}
\newcommand{\scaleA}{0.3}
\begin{tikzpicture}[scale=\scaleA]
\draw[color=gray!40] 
    (-7.3,-3.3) grid (5.3,7.3);
\fill[pattern color=oliwkowy!60, pattern=north east lines]
    (-7.3,0) -- (-3,0) -- (-3,7.3) -- (-7.3,7.3) -- cycle;
\draw[thick, color=oliwkowy!60]
    (-3,7.3) -- (-3,0) -- (-7.3,0);
\fill[pattern color=oliwkowy!60, pattern=north west lines]
    (-7.3,3.65) -- (0,0) -- (5.3,0) -- (5.3,7.3) -- (-7.3,7.3);
\draw[thick, color=oliwkowy!40]
    (-7.3,3.65) -- (0,0) -- (5.3,0);
\draw[thick, color=magenta] (-7.3,2) -- (5.3,2);
\draw[thick, color=magenta] (-7.3,1) -- (5.3,1);
\fill[pattern=north west lines, pattern color=magenta!40] (-7.3,2) -- (5.3,2) -- (5.3,1) -- (-7.3,1);
\fill[color=magenta] (1,0) circle (7pt) (2,0) circle (7pt) (-1,3) circle (7pt) (0,3) circle (7pt);
\draw[color=black]
    (-3,0) node[below]{$\scriptstyle-\lH-1$};
\fill[color=black]
    (0,0) circle (7pt) node[below]{$\scriptstyle\CO_X$};
\draw[thick, color=black]
    (-4,2) circle (7pt) node[below]{$\scriptstyle(-\xa\lV,\lV)$};
\draw[thick, color=black]
    (2.0,5.2) node[fill=white]{$\scriptstyle\Eff(X)$}
    (4.5,1.6) node[fill=white]{$\scriptstyle-\Imm(X)$}
    (-7,3.4) node[fill=white]{$\scriptstyle\CH^\lH$};
\end{tikzpicture}
\spaceA
\begin{tikzpicture}[scale=\scaleA]
\draw[color=gray!40] 
    (-7.3,-3.3) grid (5.3,7.3);
\fill[pattern color=oliwkowy!60, pattern=north east lines]
    (-7.3,0) -- (-5,0) -- (-5,7.3) -- (-7.3,7.3);
\draw[thick, color=oliwkowy!60]
    (-7.3,0) -- (-5,0) -- (-5,7.3);
\fill[pattern color=oliwkowy!60, pattern=north west lines]
    (-7.3,7.3) -- (0,0) -- (5.3,0) -- (5.3,7.3);
\draw[thick, color=oliwkowy!60]
    (-7.3,7.3) -- (0,0) -- (5.3,0);
\draw[thick, color=magenta] (-7.3,3) -- (5.3,3);
\draw[thick, color=magenta] (-7.3,1) -- (5.3,1);
\fill[pattern=north west lines, pattern color=magenta!40] (-7.3,3) -- (5.3,3) -- (5.3,1) -- (-7.3,1);
\fill[color=magenta] (1,0) circle (7pt) (2,0) circle (7pt) (3,0) circle (7pt) (4,0) circle (7pt) (2,-1) circle (7pt) (3,-1) circle (7pt) (4,-1) circle (7pt) (3,-2) circle (7pt) (4,-2) circle (7pt) (4,-3) circle (7pt);
\fill[color=magenta] (-1,4) circle (7pt) (0,4) circle (7pt) (1,4) circle (7pt) (2,4) circle (7pt) (-1,5) circle (7pt) (0,5) circle (7pt) (1,5) circle (7pt) (-1,6) circle (7pt) (0,6) circle (7pt) (-1,7) circle (7pt);
\draw[color=black]
     (-5,0) node[below]{$\scriptstyle{-\lH-1}$};
\fill[color=black]
    (0,0) circle (7pt) node[below]{$\scriptstyle\CO_X$};
\draw[thick, color=black]
    (-3,3) circle (7pt) node[below right]{$\scriptstyle(-\xa\lV,\lV)$};
\draw[thick, color=black]
    (2.5,6.2) node[fill=white]{$\scriptstyle\Eff(X)$}
    (2.5,2.0) node[fill=white]{$\scriptstyle-\Imm(X)$}
    (-7,3.4) node[fill=white]{$\scriptstyle\CH^\lH(X)$};
\end{tikzpicture}
\caption{The triple intersection $\Eff(X)\cap-\Imm(X)\cap\CH^\lH(X)$. The left-hand side displays $X(\lH,\lV;c)$ for $\lH=2=\lV=2$ and $\xa=\xb=2$. For the right-hand side we have $\lH=4$, $\lV=3$ and $\xa=1$, $\xb=2$. In particular, Proposition~\ref{prop:AcyclCrit} applies.}
\label{fig:AcyclCrit}
\end{figure}

Next we characterise \effec\ sets in terms of layers as follows. For $k\in\layers$, we call $\cl(k)$ {\em bad} if there exists $k'\in\layers$ with 
\[
k'<k,\hspace{0.8em}\cl(k')^L\not\leq_\nef\cl(k)^L\hspace{0.8em}\mbox{and}\hspace{0.8em}\cl(k')^R\not\leq_{\mathrm{eff}}\cl(k)^L.
\]
In particular, a bad layer is necessarily displaced, but the converse is false as shown by Figure~\ref{fig:BadDisplaced}.

\begin{figure}[ht]
\newcommand{\spaceA}{\hspace{10pt}}
\newcommand{\scaleA}{0.3}
\begin{tikzpicture}[scale=\scaleA]
\draw[color=gray!40] (-4.3,-1.3) grid (5.3,8.3);
\draw[thick, color=pink]
    (-1.5,7.5) -- (-1.5,3.43) -- (2.5,3.43) -- (-1.5,7.5)
    (0.2,7.1) node{$\scriptstyle\DeltaUp$};
\fill[thick, color=colorX]
    (0,4) circle (7pt) (1,4) circle (7pt)
    (0,5) circle (7pt);
\draw[thick, color=colorX]
    (0,4) circle (10pt) (3,4) circle (10pt);
\draw[thick, color=black]
    (0,0) circle (10pt) node[below]{$\scriptstyle\CO_X$};
\fill[thick, color=colorZ]
    (0,0) circle (7pt) (1,0) circle (7pt) 
    (-1,1) circle (7pt) (0,1) circle (7pt) (1,1) circle (7pt) 
    (-1,2) circle (7pt) (0,2) circle (7pt) 
    (1,2) circle (7pt)  (2,2) circle (7pt)
    (-1,3) circle (7pt) (0,3) circle (7pt) 
    (1,3) circle (7pt)  (2,3) circle (7pt);
\draw[thick, color=colorZ]
    (-2,1) circle (10pt) (1,1) circle (10pt)
    (-1,2) circle (10pt) (2,2) circle (10pt)
    (-1,3) circle (10pt) (2,3) circle (10pt);
\draw[thick, color=black]
    (-6,4) node{$\scriptstyle\cl(\infty)^L$}
    (-6,3) node{$\scriptstyle\cl(3)^L$}
    (-6,2) node{$\scriptstyle\cl(2)^L$}
    (-6,1) node{$\scriptstyle\cl(-\infty)^L$};
\end{tikzpicture}
\spaceA
\begin{tikzpicture}[scale=\scaleA]
\draw[color=gray!40] (-4.3,-1.3) grid (5.3,8.3);
\draw[very thick, color=pink]
    (-1.5,7.5) -- (-1.5,3.43) -- (2.5,3.43) -- (-1.5,7.5)
    (0.2,7.1) node{$\scriptstyle\DeltaUp$};
\fill[thick, color=colorX]
    (0,4) circle (7pt) (1,4) circle (7pt) (0,5) circle (7pt);
\draw[thick, color=colorX]
    (0,4) circle (10pt) (3,4) circle (10pt);
\fill[thick, color=colorZ]
    (0,0) circle (7pt) (1,0) circle (7pt) 
    (-1,1) circle (7pt) (0,1) circle (7pt) (1,1) circle (7pt) 
    (1,2) circle (7pt) (2,2) circle (7pt) 
    (3,2) circle (7pt)  (4,2) circle (7pt)
    (-1,3) circle (7pt) (0,3) circle (7pt) 
    (1,3) circle (7pt)  (2,3) circle (7pt);
\draw[thick, color=black]
    (0,0) circle (10pt) node[below]{$\scriptstyle\CO_X$};
\draw[thick, color=colorZ]
    (-2,1) circle (10pt) (1,1) circle (10pt)
    (1,2) circle (10pt) (4,2) circle (10pt)
    (-1,3) circle (10pt) (2,3) circle (10pt);
\draw[thick, color=black]
    (7,4) node{$\scriptstyle\cl(\infty)^R$}
    (7,3) node{$\scriptstyle\cl(3)^R$}
    (7,2) node{$\scriptstyle\cl(2)^R$}
    (7,1) node{$\scriptstyle\cl(-\infty)^R$};
\end{tikzpicture}
\caption{Examples of strongly and proper \effec\ exceptional sets on $X(3,3;(0,0,-2))$. On the left hand side no layer is displaced. On the right hand side, the layer $\cl(3)$ is bad while $\cl(\infty)$ is displaced, but not bad.}
\label{fig:BadDisplaced}
\end{figure}

\begin{proposition}
\label{prop:EffMES}
Let $X=X(\lH,\lV;c)$. A \meset\ $\cl$ is \effec\ if and only if there are no bad layers.  
\end{proposition}

\begin{proof}
We will prove the contraposition.

\medskip

Assume first that for $k\in\layers$, $\cl(k)$ is a bad layer. Taking $k'\in\layers$ with $k'<k$, $\cl(k')^L\not\leq_\nef\cl(k)^L$ and $\cl(k')^R\not\leq_{\mathrm{eff}}\cl(k)^L$, we have $\cl^L(k)_1-\cl^R(k')_1\leq-\lH-1$. Therefore, $\cl^L(k)-\cl^R(k')\in\CH^\lH(X)\cap-\Imm(X)$ so that $\cl^L(k)-\cl^R(k')\not\in\Imm(X)$. It follows that $\big(\cl^R(k'),\cl^L(k)\big)\in\pF(\cl)$ which implies that $\cl$ is not \effec.

\medskip

Conversely, assume that $\cl$ is not \effec. Then there exists a pair $(\cA,\cB)\in\pF(\cl)$ with $B-A\not\in\Eff(X)$. In particular, $A<_{\mathrm{vert}}B$ and $A-B\not\in-\Imm(X)$. It follows that $1\leq(B-A)_2\leq\lV$ and
\[
B-A\in\Z\times\Z_{\geq0}\setminus\big(\Eff(X)\cup(\Imm(X)\cap\Z\times\Z_{\geq0})\big)\subseteq\CH^\lH(X)
\]
whence $B_1-A_1\leq-\lH-1$. 

\medskip

Therefore, we cannot have $\cA$ and $\cB$ both in $\Zbound$ nor in $\cX$ whence $\cA\in\Zbound$ and $\cB\in\cZ_{\mathrm{free}}$, or $\cA\in\cZ_{\mathrm{free}}$ and $\cB\in\cX$. Put $b:=\cB_2$ if $\cB\in\cZ_{\mathrm{free}}$ and $b:=\infty$ otherwise. Similarly, $a:=\cA_2$ if $\cA\in\cZ_{\mathrm{free}}$ and $a:=-\infty$ otherwise. Then $a$, $b\in\layers$ and $a<b$. Further, $B-A\not\in\Eff(X)$ implies $\cl(a)^R\not\leq_{\mathrm{eff}}\cl(b)^L$. Finally,
\[
\cl(b)^L_1-\cl(a)^L_1-\ell=\cl(b)^L_1-\cl(a)^R_1\leq\cB_1-\cA_1\leq-\lH-1.
\]
Hence $\cl(b)$ is a bad layer.
\end{proof}

Figure~\ref{fig:HSBadLayer} illustrates an example of an \effec, yet not strongly exceptional set.

\begin{figure}
\begin{tikzpicture}[scale=0.4]
\draw[color=gray!40] 
    (-6.3,-2.3) grid (3.3,4.3);
\fill[pattern color=oliwkowy!60, pattern=north east lines]
    (-6.3,0) -- (-2,0) -- (-2,4.3) -- (-6.3,4.3);
\draw[thick, color=oliwkowy!60]
    (-6.3,0) -- (-2,0) -- (-2,4.3);
\fill[pattern color=oliwkowy!60, pattern=north west lines]
    (-6.3,3.15) -- (0,0) -- (3.3,0) -- (3.3,4.3) -- (-6.3,4.3) -- cycle;
\draw[thick, color=oliwkowy!60]
    (-6.3,3.15) -- (0,0) -- (3.3,0);
\draw[very thick, color=magenta]
    (-6.3,1) -- (3.3,1);
\fill[pattern color=yellow!40, pattern=north east lines]
    (0,4.3) -- (0,0) -- (3.3,0) -- (3.3,4.3) -- cycle;
\fill[thick, color=black]
    (0,0) circle (5pt);
\draw[thick, color=black]
    (0,0) circle (7pt) (-1,0) circle (7pt);
\draw[thick, color=black]
    (-2,1) circle (7pt) (-1,1) circle (7pt);
\fill[thick, color=black]
    (-2,1) circle (5pt) (-1,1) circle (5pt);
\draw[thick, color=black]
    (-1,2) circle (7pt) (0,2) circle (7pt);
\fill[thick, color=black]
    (-1,2) circle (5pt);
\draw[very thick, color=red]
    (-1.5,1.5) -- (-1.5,3) -- (0,1.5) -- cycle;
\draw[very thick, color=red]
    (-0.2,3.1) node{$\scriptstyle\DeltaUp$};
\draw[color=black]
    (0.1,-0.3) node[below]{$\scriptstyle\CO_X$}
    (-7.8,2) node{$\scriptstyle\cl(\infty)^L$}
    (-7.8,1) node{$\scriptstyle\cl(1)^L$}
    (-7.8,0) node{$\scriptstyle\cl(-\infty)^L$};
\draw[color=black]
    (5,2) node{$\scriptstyle\cl(\infty)^R$}
    (5,1) node{$\scriptstyle\cl(1)^R$}
    (5,0) node{$\scriptstyle\cl(-\infty)^R$};
\end{tikzpicture}
\caption{On the second Hirzebruch surface $X=X(1,1;(0,-2))$, the exceptional set given by the black points is \effec, yet not strongly exceptional: The layer $\cl(1)$ is displaced, but not bad.}
\label{fig:HSBadLayer}
\end{figure}

\section{The projective bundle of the cotangent bundle}
\label{sec:CoTang}
In this section we take $\CE=\Omega_{\P^\lH}(1)$, $\lH\geq2$, and look at the projectivisation
\[
\pi\colon X_\lH:=\P(\Omega_{\P^\lH}(1))\to\P^\lH. 
\]
The twist by $\CO_{\P^\lH}(1)$ implies that the natural coordinates on $\Pic(X)$ are nef, see Remark~\ref{rem:POmegaNatural}.

\subsection{Immaculate locus}
\label{subsec:ImmLocXl}
We first compute the immaculate locus of $X_\lH$. For the following theorem, we recall the definition of the binomial coefficient
\begin{equation}
\label{eq:binom}
{\alpha \choose k}=\frac{\alpha(\alpha-1)\cdots(\alpha-k+1)}{k!}
\end{equation}
where $\alpha\in\R$ and $k\in\Z_{\geq0}$. By convention, the empty product is one. In particular, this allows us to write
\[
\chi(\CO_{\P^\lH}(i))={i+\lH\choose\lH}
\]
for any $i\in\Z$.

\begin{theorem}
\label{thm:CohomPOmega}
Let $\CL=(i,j)$ be a line bundle on $X_\lH=\P(\Omega_{\P^\lH}(1))$ with respect to the natural coordinates induced by $\Omega_{\P^\lH}(1)$. For $j\geq0$ we have
\begin{itemize}
\item $\CL\in\CH^0(X_\lH)\iff(i,j)\in\cone{(1,0),(0,1)}$

\smallskip

\item $\CL\in\CH^{\lH-1}(X_\lH)\iff(i,j)\in(-\lH,1)+\cone{(0,1),(-1,1)}$

\smallskip

\item $\CL\in\CH^\lH(X_\lH)\iff(i,j)\in(-\lH-1,0)+\cone{(-1,1),(-1,0)}$.
\end{itemize}
For $j<0$ we have
\begin{itemize}
\item $\CL\in\CH^{\lH-1}(X_\lH)\iff(i,j)\in(-\lH,1)+\cone{(1,0),(1,-1)}$

\smallskip

\item $\CL\in\CH^{\lH}(X_\lH)\iff(i,j)\in(0,-\lH-1)+\cone{(1,-1),(0,-1)}$

\smallskip

\item $\CL\in\CH^{2\lH-1}(X_\lH)\iff(i,j)\in(-\lH,-\lH)+\cone{(-1,0),(0,-1)}$;
\end{itemize}
any other $\CH^k(X_\lH)$ is empty. In particular, at most one cohomology group is nontrivial. Furthermore, the Euler characteristic of $\CL=(i,j)$ is equal to
\[
\chi(X_\lH,\CL)=\binom{i+\lH}{\lH} \binom{j+\lH}{\lH}-\binom{i+\lH-1}{\lH} \binom{j+\lH-1}{\lH}.
\] 
\end{theorem}

\begin{corollary}
\label{coro:cohomPOmega}
$\CL = (i,j)$ is immaculate if and only if it lies in the union of the horizontal strip $\mathcal{H}:=\{(i,j)\in\Z^2\mid -\lH+1\leq j\leq-1\}$, the {\em vertical strip} $\mathcal{V}:=\{(i,j)\in\Z^2\mid-\lH+1\leq i\leq-1\}$ and the {\em anti-diagonal} $\mathcal{D}:=\{(i-\lH,-i)\in\Z^2\mid i\in\Z\}$. 
\end{corollary}

See also Figure~\ref{fig:immaculatePOmegaL} for a schematic sketch of the immaculate locus.

\begin{figure}[ht]
\begin{tikzpicture}[scale=0.4]
\draw[color=gray!40] (-7.3,-6.3) grid (4.3,4.3);
\fill[pattern=north west lines, pattern color=oliwkowy!60] (0,0) -- (0,4.3)-- (4.3,4.3) -- (4.3,0);
\fill[pattern=north east lines, pattern color=yellow!40] (0,0) -- (0,4.3)-- (4.3,4.3) -- (4.3,0);
\draw[color=black,inner sep=0] (5,2) node[fill=white] {$\scriptstyle\Eff(X)=\mr{Nef}(X)$};
\draw[->,color=black] (0,0) -- (5.3,0) node[right]{$\scriptstyle h$};
\draw[->,color=black] (0,0) -- (0,5.3) node[above]{$\scriptstyle H$};
\fill[thick, color=black] (0,0) circle (3pt) node[below]{$\scriptstyle\CO$};
\draw[very thick, color=magenta] (-7.3,-1) -- (-2,-1);
\draw[very thick, color=magenta] (-1,-1) -- (4.3,-1);
\draw[very thick, color=magenta] (-7.3,-2) -- (-2,-2);
\draw[very thick, color=magenta] (-1,-2) -- (4.3,-2);
\fill[pattern=north west lines, pattern color=magenta!40] (-7.3,-2) rectangle (4.3,-1);
\draw[very thick, color=magenta] (-2,-6.3) -- (-2,-2);
\draw[very thick, color=magenta] (-2,-1) -- (-2,4.3);
\draw[very thick, color=magenta] (-1,-6.3) -- (-1,-2);
\draw[very thick, color=magenta] (-1,-1) -- (-1,4.3);
\fill[pattern=north west lines, pattern color=magenta!40] (-2,-6.3) -- (-2,4.3) -- (-1,4.3) -- (-1,-6.3);
\draw[very thick, color=magenta] (-7.3,4.3) -- (-2,-1) (-1,-2) -- (3.3,-6.3);
\draw[very thick, color=oliwkowy!60] (0,0) -- (4.3,0);
\draw[very thick, color=oliwkowy!60] (0,0) -- (0,4.3);
\draw[very thick, color=oliwkowy!60] (-3,1) -- (-3,4.3);
\draw[very thick, color=oliwkowy!60] (-3,1) -- (-6.3,4.3);
\draw[color=black,inner sep=0]  (-3.5,4.5) node[fill=white] {$\scriptstyle\CH^{\lH-1}$};
\fill[pattern=vertical lines, pattern color=oliwkowy!60] (-3,1) -- (-6.3,4.3)-- (-3,4.3) -- cycle;
\draw[very thick, color=oliwkowy!60] (-4,0) -- (-7.3,3.3);
\draw[very thick, color=oliwkowy!60] (-4,0) -- (-7.3,0);
\fill[pattern=north east lines, pattern color=oliwkowy!60] (-4,0) -- (-7.3,0) -- (-7.3,3.3) -- cycle;
\draw[color=black,inner sep=0] (-7.5,1.5) node[fill=white] {$\scriptstyle\CH^\lH$};
\draw[very thick, color=oliwkowy!60] (-3,-3) -- (-7.3,-3);
\draw[very thick, color=oliwkowy!60] (-3,-3) -- (-3,-6.3);
\fill[thick, color=black] (-3,-3) circle (3pt) node[above]{$\scriptstyle K_X$};
\fill[pattern=north east lines, pattern color=oliwkowy!60] (-7.3,-3) -- (-3,-3) -- (-3,-6.3) -- (-7.3,-6.3);
\draw[color=black,inner sep=0] (-7.9,-4.5) node[fill=white] {$\scriptstyle\CH^{2\lH-1}$};
\draw[very thick, color=oliwkowy!60] (0,-4) -- (0,-6.3);
\draw[very thick, color=oliwkowy!60] (0,-4) -- (2.3,-6.3);
\fill[pattern=vertical lines, pattern color=oliwkowy!60] (0,-4) -- (0,-6.3) -- (2,-6.3) -- cycle;
\draw[color=black,inner sep=0] (1.5,-6.5) node[fill=white] {$\scriptstyle\CH^\lH$};
\draw[very thick, color=oliwkowy!60] (1,-3) -- (4.3,-3);
\draw[very thick, color=oliwkowy!60] (1,-3) -- (4.3,-6.3);
\fill[pattern=north west lines, pattern color=oliwkowy!60] (1,-3) -- (4.3,-3) -- (4.3,-6.3) -- cycle;
\draw[color=black,inner sep=0] (5.1,-4.5) node[fill=white] {$\scriptstyle\CH^{\lH-1}$};
\end{tikzpicture}
\caption{The immaculate locus of $X_\lH = \P(\Omega_{\P^\lH}(1))$ (drawn in magenta for $\lH=3$).}
\label{fig:immaculatePOmegaL}
\end{figure}

\medskip

\begin{proof}[Proof of Theorem~\ref{thm:CohomPOmega}]
Since $X_\lH=\P(\Omega_{\P^\lH})\cong\Flag(1,\lH,\lH+1)$, the theorem of Borel-Bott-Weil (see for instance~\cite{brionFlag} or~\cite{luisFlag}) immediately implies that there is at most one nontrivial cohomology group of $\CL$. In particular, up to a sign its dimension coincides with the Euler characteristic of $\CL$.

\medskip

Let now $\CL=(i,j)$ be a line bundle on $X_\lH$ with respect to natural coordinates. By Proposition~\ref{prop:CohomPBundleP}, $\CL$ is immaculate if $-\lH+1\leq j\leq-1$. Furthermore, $H^\bullet(X,(i,0))=H^\bullet(\P^\lH,\CO(i))$ so that $(i,0)$ lies in $\CH^\lH$ if $i\leq-\lH-1$, and in $\CH^0$ if $i\geq0$. This settles the range $-\lH+1\leq j\leq0$.

\medskip

Next assume that $j\leq-\lH$. Recall from Remark~\ref{rem:CLB} that the canonical line bundle of $\P(\CE)$ is $K_{\P(\CE)}=(-\lH-1-e)h-(\lV+1)H$ where $\lV+1=\rk\CE$ and $\CO(e)=\det(\CE)$. Hence
\[
K_{X_\lH}=(-\lH,-\lH)
\]
in the natural coordinates on $X_\lH$. Serre duality thus implies
\begin{align*}
H^\bullet(X_\lH,\CL)&=H^\bullet\big(X_\lH,(i,j)\big)\\
&=H^{2\lH-1-\bullet}(X_\lH,\CL^\vee\otimes K_{X_\lH})^\vee=H^{2\lH-1-\bullet}\big(X_\lH,(-i-\lH,-j-\lH)\big)^\vee.
\end{align*}
Since $-j-\lH\geq0$ we are reduced to the case of positive $j$.

\medskip

We therefore assume that $j\geq1$ and consider the dual Euler sequence
\[
0\to\CO\to\CO(1)^{\oplus(\lH+1)}\to\CT_{\P^\lH}\to0
\]
on $\P^\lH$. Taking the $j$-th symmetric power yields
\[
0\to\CO\otimes\Sym^{j-1}(\CO(1)^{\oplus(\lH+1)})\to\Sym^j(\CO(1)^{\oplus(\lH+1)})\to\Sym^j\CT_{\P^\lH}\to0
\]
and consequently
\begin{equation}
\label{eq:SymCT}
0\to\CO(j-1)^{\oplus{j+\lH-1\choose j-1}}\to\CO(j)^{\oplus{j+\lH\choose j}}  \to\Sym^j\CT_{\P^\lH}\to0.
\end{equation}
By Proposition~\ref{prop:CohomPBundleP}, $\pi_*(\CO(H))=\CT_{\P^\lH}(-1)$ whence
\[
\pi_*\CL=\pi_*(i,j)=\CO(i)\otimes\Sym^j(\CT_{\P^\lH}(-1))=\CO(i-j)\otimes\Sym^j\CT_{\P^\lH}
\]
and thus
\begin{equation}
\label{eq:piL}
0\to\CO(i-1)^{\oplus{j+\lH-1\choose \lH}}\to\CO(i)^{\oplus{j+\lH\choose \lH}}\to\pi_*\CL\to0.
\end{equation}
In particular, the Euler characteristic of $\CL=(i,j)$ is
\begin{equation}
\label{eq:chi}
\begin{split}
\chi(X_\lH,\CL)
&=\chi(\P^\lH,\pi_*\CL)=\chi(\P^\lH,\CO(i))\cdot{j+\lH\choose \lH}-\chi(\P^\lH,\CO(i-1))\cdot{j+\lH-1\choose \lH}\\
&=\binom{i+\lH}{\lH} \binom{j+\lH}{\lH} - \binom{i+\lH-1}{\lH} \binom{j+\lH-1}{\lH}
\end{split}
\end{equation}
We are now in a position to determine the loci of non-vanishing cohomology and distinguish three cases.

\smallskip

{\em Case $i\geq0$.}
As ${j+\lH\choose\lH}>0$, the sheaf in the middle of~\eqref{eq:piL} has nontrivial global sections, and so does therefore $\CL=(i,j)$.

\smallskip

{\em Case $-\lH+1\leq i\leq-1$.}
Here, $(i,j)$ is immaculate for so are the other sheaves in~\eqref{eq:piL}.

\smallskip

{\em Case $i\leq-\lH$.}
The long exact sequence of cohomology reads
\begin{align*}
0\to H^{\lH-1}(X_\lH,\CL) 
&\to H^\lH(\P^\lH,\CO(i-1))^{\oplus{j+\lH-1\choose \lH}}\to \\
&\to H^\lH(\P^\lH,\CO(i))^{\oplus{j+\lH\choose \lH}}
\to H^\lH(X_\lH,\CL)\to0.
\end{align*}
Since for any nonnegative integer $k$ and real $\alpha$, we have
\[
\alpha {\alpha-1 \choose k}=\frac{\alpha(\alpha-1)\cdots(\alpha-k+1)(\alpha-k)}{k!} = (\alpha-k)\cdot{\alpha\choose k}.
\]
and so
\[
\chi(X_\lH,\CL)=\binom{i+\lH-1}{\lH} \binom{j+\lH}{\lH} \left(\frac{i+\lH}i - \frac{j}{j+\lH} \right),
\]
cf.\ Equation~\eqref{eq:chi}. The sign of the first two factors depends only on $\lH$, whereas the sign of the third factor flips at the line
\[
i+j=-\lH,
\]
where the line bundles are immaculate. From here, all the statements easily follow.
\end{proof}

\begin{remark}
\label{rem:POmegaNatural}
We note that the natural coordinates on $\Pic(X_\lH)$ are actually nef. Indeed,
\[ 
\pi_*\CO(H)=\CT_{\P^\lH}(-1)=\CO^{\oplus(\lH+1)}\!/\CO(-1) 
\]
is nef as a quotient of a nef bundle, and therefore stays nef when pulled back to $X_\lH$~\cite[Thm II.6.2.12]{lazarsfeld-bundles}. Hence $H$ is nef so that $\Nef(X)=\Eff(X)=\cone{(1,0),(0,1)}$ by Remark~\ref{rem:EffNef}.
\end{remark}

To compute the immaculate locus of $X_\lH^\vee=\P(\Tang_{\P^\lH}(-2))$
we can proceed similarly and obtain the following theorem.

\begin{theorem}
\label{thm:cohomPTangL}
Let $\CL=(i,j)$ be a line bundle on $X_\lH^\vee=\P(\CT_{\P^\lH}(-2))\to\P^\lH$ given with respect to the natural coordinates of $\CT_{\P^\lH}(-2)$. For $j\geq 0$ we have
\begin{itemize}
\item $\CL\in\CH^0(X_\lH^\vee)\iff(i,j)\in\cone{(1,0),(1,0)}$

\smallskip

\item $\CL\in\CH^1(X_\lH^\vee)\iff(i,j)\in(-2,1)+\cone{(0,1),(-1,1)}$

\smallskip

\item $\CL\in\CH^\lH(X_\lH^\vee)\iff(i,j)\in(-\lH-1,0)+\cone{(-1,1),(-1,0)}$
\end{itemize}
and for $j<0$
\begin{itemize}
\item $\CL\in\CH^{\lH-1}(X_\lH^\vee)\iff(i,j)\in(\lH-1,-\lH)+\cone{(1,0),(1,-1)}$

\smallskip

\item $\CL\in\CH^{2\lH-2}(X_\lH^\vee)\iff(i,j)\in(0,-\lH-1)+\cone{(1,-1),(0,-1)}$

\smallskip

\item $\CL\in\CH^{2\lH-1}(X_\lH^\vee)\iff(i,j)\in(-\lH+1,-\lH)+\cone{(-1,0),(0,-1)}$;
\end{itemize}
any other $\CH^k(X_\lH)$ is empty. In particular, at most one cohomology group is nontrivial. Furthermore, the Euler characteristic of $\CL=(i,j)$ is equal to
\[
\chi(X_\lH^\vee,\CL)=
{i+j+\lH\choose\lH}{j+\lH\choose\lH}-{i+j+1+\lH\choose\lH}{j+\lH-1\choose\lH} .
\]
\end{theorem}

\begin{proof}
Since $X_\lH^\vee\cong\Flag(1,2,\lH+1)$ we can again appeal to the theorem of Borel-Bott-Weil and deduce that at most one cohomology group $H^k(X_\lH^\vee,\CL)$ is non-trivial. Now for $j\geq0$ we have
\[
\pi_*\CL=\CO(i)\otimes\Sym^j(\Omega(2))=\CO(i+2j)\otimes\Sym^j\Omega
\]
on $\P^\lH$. This fits into the sequence
\[
0\to\Sym^j\Omega\to\CO(-j)^{\oplus{j+\lH\choose j}}\to\CO(1-j)^{\oplus{j+\lH-1\choose j-1}}\to0
\]
obtained by dualising~\eqref{eq:SymCT} if we tensor with $\CO(i+2j)$, that is,
\[
0\to\pi_*\CL\to\CO(i+j)^{\oplus{j+\lH\choose\lH}}\to\CO(i+j+1)^{\oplus{j+\lH-1\choose\lH}}\to0.
\]
Consequently, the Euler characteristic of $\CL$ is
\[
\chi(X_\lH^\vee,\CL) = \chi(\P^\lH,\pi_* \CL) = 
{i+j+\lH \choose \lH} {j+\lH \choose \lH} - {i+j+1+\lH \choose \lH} {j-1+\lH \choose \lH} .
\]
The result then follows as for Theorem~\ref{thm:CohomPOmega}.
\end{proof}

See Figure~\ref{fig:immaculatePTangL} for a schematic sketch of the immaculate locus.

\begin{figure}[ht]
\begin{tikzpicture}[scale=0.4]
\draw[color=gray!40] (-7.3,-6.3) grid (4.3,4.3);
\fill[pattern=north west lines, pattern color=oliwkowy!60] (0,0) -- (0,4.3)-- (4.3,4.3) -- (4.3,0);
\fill[pattern=north east lines, pattern color=yellow!40] (0,0) -- (0,4.3)-- (4.3,4.3) -- (4.3,0);
\draw[color=black,inner sep=0] (5,2) node[fill=white] {$\scriptstyle\Eff(X)=\mr{Nef}(X)$};
\draw[->,color=black] (0,0) -- (5.3,0) node[right]{$\scriptstyle h$};
\draw[->,color=black] (0,0) -- (0,5.3) node[above]{$\scriptstyle H$};
\draw[very thick, color=oliwkowy!60] (0,0) -- (4.3,0);
\draw[very thick, color=oliwkowy!60] (0,0) -- (0,4.3);
\fill[thick, color=black] (0,0) circle (3pt) node[below]{$\scriptstyle\CO$};
\draw[very thick, color=magenta] (-7.3,-1) -- (-2,-1) (-1,-1) -- (4.3,-1);
\draw[very thick, color=magenta] (-7.3,-2) -- (-1,-2) (0,-2) -- (4.3,-2);
\draw[very thick, color=magenta] (-7.3,4.3) -- (-2,-1) (-6.3,4.3) -- (-1,-1);
\draw[very thick, color=magenta] (-1,-2) -- (3.3,-6.3) (0,-2) -- (4.3,-6.3);
\draw[very thick, color=magenta] (-1,4.3) -- (-1,-1) (-1,-2) -- (-1,-6.3);
\fill[pattern=north west lines, pattern color=magenta!40] (-7.3,-2) rectangle (4.3,-1);
\fill[pattern=north west lines, pattern color=magenta!40] (-7.3,4.3) -- (-6.3,4.3) -- (4.3,-6.3) -- (3.3,-6.3);
\fill[pattern=vertical lines, pattern color=oliwkowy!60] (-2,1) -- (-5.3,4.3)-- (-2,4.3) -- cycle;
\draw[very thick, color=oliwkowy!60] (-2,1) -- (-5.3,4.3);
\draw[very thick, color=oliwkowy!60] (-2,1) -- (-2,4.3);
\draw[color=black,inner sep=0] (-3.5,4.5) node[fill=white] {$\scriptstyle\CH^1$};
\fill[pattern=north east lines, pattern color=oliwkowy!60] (-4,0) -- (-7.3,3.3)-- (-7.3,0) -- cycle;
\draw[very thick, color=oliwkowy!60] (-4,0) -- (-7.3,3.3);
\draw[very thick, color=oliwkowy!60] (-4,0) -- (-7.3,0);
\draw[color=black,inner sep=0] (-7.5,1.5) node[fill=white] {$\scriptstyle\CH^\lH$};
\fill[pattern=north east lines, pattern color=oliwkowy!60] (-7.3,-3) -- (-2,-3)-- (-2,-6.3) -- (-7.3,-6.3);
\draw[very thick, color=oliwkowy!60] (-2,-3) -- (-7.3,-3);
\draw[very thick, color=oliwkowy!60] (-2,-3) -- (-2,-6.3);
\fill[thick, color=black] (-2,-3) circle (3pt) node[above left]{$\scriptstyle K_X$};
\draw[color=black,inner sep=0] (-7.9,-4.5) node[fill=white] {$\scriptstyle \CH^{2\lH-1}$};
\fill[pattern=vertical lines, pattern color=oliwkowy!60] (0,-6.3) -- (0,-4)-- (2.3,-6.3) -- cycle;
\draw[very thick, color=oliwkowy!60] (0,-4) -- (0,-6.3);
\draw[very thick, color=oliwkowy!60] (0,-4) -- (2.3,-6.3);
\draw[color=black,inner sep=0] (1.75,-6.5) node[fill=white] {$\scriptstyle \CH^{2\lH-2}$};
\fill[pattern=north west lines, pattern color=oliwkowy!60] (4.3,-3) -- (2,-3)-- (4.3,-5.3) -- cycle;
\draw[very thick, color=oliwkowy!60] (2,-3) -- (4.3,-3);
\draw[very thick, color=oliwkowy!60] (2,-3) -- (4.3,-5.3);
\draw[color=black,inner sep=0] (5.1,-4.5) node[fill=white] {$\scriptstyle\CH^{\lH-1}$};
\end{tikzpicture}
\caption{The immaculate locus of $X_\lH^\vee=\P(\Tang_{\P^\lH}(-2))$ (drawn in magenta for $\lH=3$).}
\label{fig:immaculatePTangL}
\end{figure}

\begin{remark}
Recall that under the identification
\[
\P(\Omega_{\P^\lH}(1))\cong\P(\Omega_{\P^\lH})\cong\Flag(1,\lH,\lH+1),  
\]
the projection $X_\lH\to\P^\lH$ becomes the natural map $\Flag(1,\lH,\lH+1)\to\Flag(1,\lH+1)=\P^\lH$. For $\lH>2$ we get the diagram
\[
\begin{tikzcd}[column sep=-0.55ex]
&& \Flag(1,2,\lH,\lH+1)\ar[lldd, "p"'] \ar[rrdd, "q"] &&\\
&& V_1 \subset V_2 \subset V_\lH \subset \kk^{\lH+1} \ar[u, phantom, "\rotatebox{90}{$\in$}" description] \ar[ld, mapsto] \ar[rd, mapsto]\\
\P\big(\Tang_{\P^\lH}(-2)\big) \ar[dd] \ar[rrdd] & V_1\!\subset\!V_2\!\subset \kk^{\lH+1} \ar[l, phantom, "\ni" description] & & V_1\!\subset\!V_{\lH}\!\subset \kk^{\lH+1} \ar[r, phantom, "\in" description]&\,  \P\big(\Omega_{\P^\lH}(1)\big) \ar[dd] \ar[lldd] \\
\\
\Gr(2,\lH+1) && \P^\lH  && \check\P^\lH
\end{tikzcd}
\]
where we use $\Gr(2,\lH+1) = \Flag(2,\lH+1)$, $\P^\lH = \Flag(1,\lH+1)$ and $\check\P^\lH = \Flag(\lH,\lH+1)$. Both $p$ and $q$ are projective bundle maps. In particular, 
\[
p^*\colon\CD(\P(\Omega_{\P^\lH}(1)))\to\CD(\Flag(1,2,\lH,\lH+1))
\]
and
\[
q^*\colon\CD(\P(\Tang_{\P^\lH}(-2)))\to\CD(\Flag(1,2,\lH,\lH+1))
\]
are fully faithful. This implies that the (im)maculate loci of $\Pic(X_\lH)$ and $\Pic(X^\vee_\lH)$ can be obtained as hyperplanes of $\Pic(\Flag(1,2,\lH,\lH+1))$. 
Regarding $\Flag(1,2,\lH,\lH+1)$ as a projective bundle over $\Gr(2,\lH+1)$, $\P^\lH$ and $\check\P^\lH$ respectively, this accounts for the three immaculate strips which occur in $\Pic(X_\lH)$ and $\Pic(X^\vee_\lH)$.
\end{remark}

\subsection{The associated lattice}
It is useful to reduce $\Pic(X)$ further by taking the quotient with an {\em admissible} sublattice $\Lambda$, that is,
\[
\Lambda\cap(\Imm(X)\cup-\Imm(X))=\varnothing.
\]
This is motivated by the following

\begin{lemma}
\label{lem:PicLambda}
If $\Lambda$ is an admissible sublattice of $\Pic(X)$, then for any exceptional set $s$ the composition
\[
\cl\into\Pic(X)\onto\Pic(X)/\Lambda
\]
is injective.  
\end{lemma}

\begin{proof}
$\CL$, $\CL'\in s$ are mapped to the same point in $\Pic(X)/\Lambda$ if $\CL_i-\CL_j\in\Lambda$. But $\CL_i-\CL_j\in\Imm(X)\cup-\Imm(X)$, contradiction.
\end{proof}

In the following we let $\phi_\Lambda \colon \Pic(X) \to \Pic(X)/\Lambda$ be the natural projection.

\begin{example}
\label{exam:LatticeTV}
For the toric case, $\Lambda=\langle K_X,\pi^*K_{\P^\lH}\rangle$ is admissible, cf.\ also~\cite[Subsection $($3.4$)$]{fesfano}. In this case, $\Pic(X)/\Lambda$ contains exactly $\rk K_0(X)$ elements, so that an exceptional set of line bundles $\cl$ is maximal if and only if $\phi_\Lambda$ restricted to $\cl$ is a bijection.
\end{example}

However, the lattice in Example~\ref{exam:LatticeTV} is no longer admissible for $X=X_\lH$: $K_X\otimes\pi^*\omega_{\P^\lH}^{-1}$ lies in $\Imm(X)$. On the other hand, $X$ has two projections $\pi_1$ and $\pi_2$ onto $\P^\lH$. Then
\[
\Lambda=\langle\pi_1^*K_{\P^\lH},\pi_2^*K_{\P^\lH}\rangle 
\]
is admissible and formally agrees with the lattice of the product case since $K_{\P^\lH\times\P^\lV}=\pi_1^*K_{\P^\lH}\otimes\pi_2^*K_{\P^\lV}$. In particular, any exceptional set $\cl$ injects into $\Pic(X)/\Lambda\cong\big(\Z/(\lH+1)\Z\big)^2$.

\begin{remark}
\label{rem:otherlattice}
Other choices for admissible lattices are possible, for instance, $\Lambda'=\langle K_{X},\CO_X\big((\lH+2)h\big)\rangle$.
\end{remark}

Unlike the toric case we do not get a bijection since $\rk K_0(X)=\lH(\lH+1)$. There are, however, constraints on the image of $\Phi_\Lambda$. Let 
\[
\Delta:=\{(a,a)\mid a\in\Z/(\lH+1)\Z\}
\]
and consider the ``shifted diagonals'' $p+\Delta$ for $p\in\Pic(X)/\Lambda$.

\begin{proposition}
Let $\cl$ be an exceptional set on $X=X_\lH$. The image $\Phi_\Lambda(\cl)$ consists of at most $\lH(\lH+1)$ points where none of the shifted diagonals $p+\Delta$ contains more than $\lH$ points.

\medskip

Furthermore, for a \meset\ $\cl$ the intersection of $\Phi(\cl)$ with each shifted diagonal in $\Pic(X)/\Lambda$ consists, in this order, of the points
\[
p+(1,1),\; p+(2,2),\ldots,p+(\lH,\lH)
\]
for some uniquely determined {\em gap point $p$}, namely the missing point of the shifted diagonal $p+\Delta$. 
\end{proposition}

\begin{proof}
The set $\Phi_\Lambda(\Imm(X_\lH))$ equals $\Pic(X)/\Lambda\setminus\{\Phi(\CO_X),\,\Phi(K_X)\}$. Thus, if $\CL_i$, $\CL_j\in\cl$ are elements of an exceptional sequence with $\pi(\CL_j)=\pi(\CL_i)+(1,1)$, then necessarily $i<j$.

\medskip

Consequently, a shifted diagonal $p+\Delta$ cannot entirely belong to an exceptional sequence $\cl$ since it is periodic. Therefore, we necessarily have gaps; for a \meset\ $\cl$, there cannot be more than one. 

\medskip

Let then $p\in\Pic(X)/\Lambda$ be this unique $\cl$-gap in its own shifted diagonal $p+\Delta$. The remaining elements are
\[
p+(1,1),\;p+(2,2),\ldots,p+(\lH,\lH)
\]
which belong to the image of $\cl$ under $\Phi$, and any element in $\Phi^{-1}_\Lambda\big(p+\nu\cdot(1,1)\big)\cap\cl$ necessarily precedes any element in $\Phi^{-1}_\Lambda\big(p+(\nu+1)\cdot(1,1)\big)$.
\end{proof}

\begin{remark}
The proposition can be also regarded as a consequence of the admissibility of $\Lambda'$ in Remark~\ref{rem:otherlattice}.
\end{remark}

\subsection{Maximal exceptional sequences on $X_2=\P(\Omega_{\P^2}(1))$}
\label{sec:thmAforX2}
As discussed in Subsection~\eqref{subsec:ImmLocXl}, we can subdivide the immaculate locus into four parts: The common intersection of the (anti-)diagonal and the horizontal and vertical line; and the diagonal / horizontal / vertical line minus that common intersection. Inside the quotient $\Pic(X_2)/\Lambda$ this is reflected by the disjoint union of equivalence classes 
\begin{itemize}
\item $\{(2,0),\,(0,2)\}$ coming only from points on the diagonal
\item $\{(0,1),\,(2,1)\}$ coming only from points in the horizontal line 
\item $\{(1,0),\,(1,2)\}$ coming from points in the vertical line
\item $\{(1,1)\}$ coming from the intersection of all three parts.
\end{itemize}

\medskip

We use colours for the three basic building blocks of the immaculate locus: Blue represents the diagonal, red the horizontal and green the vertical strip, see Figure~\ref{fig:colNegImmLocFlags} for illustration. We refer to a point in a \meset\ as a blue, green or red point if its image in $\Pic(X_2)/\Lambda$ is a point of corresponding colour. As the image of a \meset\ consists of six points, at least one colour occurs necessarily twice.

\begin{figure}[htpb]
\definecolor{colorI}{RGB}{210,40,90}
\definecolor{colorJ}{RGB}{10,90,50}
\newcommand{\spaceA}{\hspace*{0.3em}}
\newcommand{\links}{-1.3}
\newcommand{\rechts}{3.3}
\newcommand{\unten}{-1.3}
\newcommand{\oben}{3.3}
\hspace*{\fill}
\begin{tikzpicture}[scale=0.45]
\draw[color=oliwkowy!40] (\links,\unten) grid (\rechts,\oben);
\fill[thick, color=black]
  (0,0) circle (4pt) node[below]{$\scriptstyle\CO_{X_2}$};
\draw[ultra thick, color=darkgreen]
  (1,3) -- (1,-1);
\draw[ultra thick, color=red]
  (3,1) -- (-1,1);
\draw[ultra thick, color=blue]
  (-1,3) -- (3,-1);
\fill[color=white]
  (1,1) circle (3pt);
\draw[thick, color=black]
  (1,1) circle (4pt);
\end{tikzpicture}
\hspace*{\fill}
\begin{tikzpicture}[scale=0.55]
\draw[color=oliwkowy!40] (\links+1,\unten+1) grid (\rechts-1,\oben-1);
\draw[thick, color=black]
  (1,1) circle (4pt);
\fill[thick, color=black]
  (0,0) circle (4pt);
\fill[color=darkgreen]
  (1,0) circle (4pt) 
  (1,2) circle (4pt);
\fill[color=red]
  (0,1) circle (4pt) 
  (2,1) circle (4pt);
\fill[color=blue]
  (2,0) circle (4pt) (0,2) circle (4pt);
\fill[thick, color=blue!20]
  (-1,0) circle (4pt) (3,2) circle (4pt);
\draw[thick, color=blue]
  (-1,0) circle (4pt) (3,2) circle (4pt);
\draw[thick, color=red]
  (-1.4,1) node{$\scriptstyle(0,1)$} (3.4,1) node{$\scriptstyle(2,1)$};
\draw[thick, color=blue]
  (-1.4,2) node{$\scriptstyle(0,2)$} (3.4,0) node{$\scriptstyle(2,0)$} 
  (4.2,2) node{$\scriptstyle(0,2)$} (-2.2,0) node{$\scriptstyle(2,0)$};
\draw[thick, color=darkgreen]
  (1,3) node{$\scriptstyle(1,2)$};
\draw[thin, color=black]
  (-0.5,-0.5) -- (2.5,-0.5) -- (2.5,2.5) -- (-0.5,2.5) -- cycle;
\draw[thin, color=intOrange]
  (0.2,0.2) -- (0.8,0.8) (1.2,1.2) -- (1.9,1.9) (2.1,2.1) -- (2.65,2.65)
  (-0.8,0.2) -- (-0.2,0.8) (0.2,1.2) -- (0.8,1.8)
  (1.2,0.2) -- (1.8,0.8) (2.2,1.2) -- (2.8,1.8)
  (3,3) node{$\scriptstyle\Delta$};
\end{tikzpicture}
\hspace*{\fill}
\caption{The coloured immaculate loci $\nImm(X_2)\subseteq\Z^2$ 
and $\Phi_\Lambda(\nImm)\subseteq(\Z/3\Z)^2$
}
\label{fig:colNegImmLocFlags}
\end{figure}

\begin{theorem}
Up to helixing and changing the order of mutually orthogonal elements, any \mes\ is either a vertically or horizontally lex ordered \mes.
\end{theorem}

\begin{proof}
Assume that $\cl=(\cl_1,\ldots,\cl_6)$, $\cl_i=(a_i,b_i)$, is a \mes\ which is not vertically lex ordered. Then there exists $1\leq i\leq5$ with $b_{i+1}-b_i<0$. If $\cl_{i+1}-\cl_i\in-\Imm\cap\Imm$, then mutation just exchanges $\cl_i$ and $\cl_{i+1}$ in the sequence. Therefore, there must exist $1\leq i\leq5$ such that either 
\begin{enumerate}
\item $\cl_{i+1}-\cl_i=(k+2,-k)$, $k\geq2$;
\item $\cl_{i+1}-\cl_i=(1,-k)$, $k\geq4$;
\item $\cl_{i+1}-\cl_i=(1,-2)$,
\end{enumerate}
cf.\ Figure~\ref{fig:AllowImm}. Helixing and normalising if necessary we may assume that $i=1$ and $\cl_1=\CO_X$.

\begin{figure}[ht]
\begin{tikzpicture}[scale=0.4]
\draw[color=gray!40] (-4.3,-4.3) grid (4.3,4.3);
\draw[color=darkgreen, very thick] (1,4.3)--(1,-4.3);
\draw[color=red, very thick] (-4.3,1)--(4.3,1);
\draw[color=blue, very thick] (-2.3,4.3)--(4.3,-2.3);
\fill[color=black] (-1,1) circle (5pt) (-3,1) circle (5pt) (1,-1) circle (5pt) (1,-3) circle (5pt) (-1,3) circle (5pt) (3,-1) circle (5pt);
\draw[color=black] (0,0) circle (5pt) node[below]{$\scriptstyle\CO_X$};
\end{tikzpicture}
\caption{The set $-\Imm\setminus(-\Imm\cap\Imm)$. The points in $-\Imm\cap\Imm$ are displayed in black.}
\label{fig:AllowImm}
\end{figure}

\medskip

{\sc Case 1.} Assume that $\cl_2=(2+k,-k)$, $k\geq2$. The set of possible successors is given by
\[
\mr{Suc}(\cl_2)=\{(1,3),\,(1,1-k),\,(3,1),\,(3+k,1),\,(1+k,1-k),\,(3+k,-1-k)\}. 
\]
Existence of at least one pair of coloured points requires $k=2$, see Figure~\ref{fig:SucBlueLine} in which case we have two red points. However, these exclude any green point, contradiction. Hence Case 1 does not occur.

\begin{figure}[ht]
\begin{tikzpicture}[scale=0.4]
\draw[color=gray!40] (-1.3,-4.3) grid (6.3,3.3);
\draw[color=darkgreen, thick] (1,3.3)--(1,-4.3);
\draw[color=red, thick] (-1.3,1)--(6.3,1);
\draw[color=blue, thick] (-1.3,3.3)--(6.3,-4.3);
\fill[color=black] (1,-1) circle (5pt) (5,-3) circle (5pt) (3,-1) circle (5pt);
\fill[color=black] (5,1) circle (5pt) (1,3) circle (5pt) (3,1) circle (5pt);
\node[shape=circle,draw,fill=white,inner sep=1pt] at (0,0) {$\scriptstyle1$};
\node[shape=circle,draw,fill=white,inner sep=1pt] at (4,-2) {$\scriptstyle2$};
\end{tikzpicture}
\caption{The set of successors $\mr{Suc}(4,-2)$ (black points).}
\label{fig:SucBlueLine}
\end{figure}

\bigskip

{\sc Case 2.} Assume that $\cl_2=(1,-k)$, $k\geq4$. The set of possible successors is given by
\[
\mr{Suc}(\cl_2)=\{(1,1-k),\,(1,2-k),\,(2-k,1),\,(2,1),\,(2,0),\,(1+k,1-k)\}, 
\]
see Figure~\ref{fig:SucGreenLine}.

\begin{figure}[ht]
\begin{tikzpicture}[scale=0.4]
\draw[color=gray!40] (-2.3,-4.3) grid (6.3,3.3);
\draw[color=darkgreen, thick] (1,3.3)--(1,-4.3);
\draw[color=red, thick] (-2.3,1)--(6.3,1);
\draw[color=blue, thick] (-1.3,3.3)--(6.3,-4.3);
\fill[color=black] (1,-2) circle (5pt) (1,-3) circle (5pt) (-2,1) circle (5pt);
\fill[color=black] (5,-3) circle (5pt) (2,0) circle (5pt) (2,1) circle (5pt);
\node[shape=circle,draw,fill=white,inner sep=1pt] at (0,0) {$\scriptstyle1$};
\node[shape=circle,draw,fill=white,inner sep=1pt] at (1,-4) {$\scriptstyle2$};
\end{tikzpicture}
\caption{The set of successors $\mr{Suc}(1,-4)$ (black points).}
\label{fig:SucGreenLine}
\end{figure}
Here, we cannot have a pair of red or blue points. Therefore, we must have two green points which results in the sequences
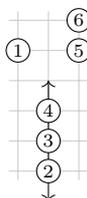
\[
\begin{tikzpicture}[scale=0.4]
\draw[color=gray!40] (-0.3,-4.3) grid (2.3,1.3);
\draw[<->] (1,-1) -- (1,-5);
\drawExSec{0,0;1,-4;1,-3;1,-2;2,0;2,1};
\end{tikzpicture}
\]
These are horizontally ordered.

\bigskip

{\sc Case 3.} Finally, assume that $\cl_2=(1,-2)$. The set of possible successors is given by
\[
\mr{Suc}(1,-2)=\{(1,-1),\,(1,0),\,(0,1),\,(2,1),\,(2,0),\,(3,-1)\}, 
\]
cf.\ Figure~\ref{fig:Suc2-1}. Note that a possible blue point on $(3,-1)$ cannot occur in a sequence with either the red point $(0,1)$ or the green point $(1,0)$ and the red point $(2,1)$. On the other hand, a red point on $(0,1)$ prevents a blue point on $(3,-1)$; together with a green point on $(1,-1)$ it prevents a blue point in $(2,0)$.

\begin{figure}[ht]
\begin{tikzpicture}[scale=0.4]
\draw[color=gray!40] (-1.3,-3.3) grid (4.3,3.3);
\draw[color=darkgreen, thick] (1,3.3)--(1,-3.3);
\draw[color=red, thick] (-1.3,1)--(4.3,1);
\draw[color=blue, thick] (-1.3,3.3)--(4.3,-2.3);
\fill[color=black] (0,1) circle (5pt) (1,0) circle (5pt) (1,-1) circle (5pt);
\fill[color=black] (2,1) circle (5pt) (2,0) circle (5pt) (3,-1) circle (5pt);
\node[shape=circle,draw,fill=white,inner sep=1pt] at (0,0) {$\scriptstyle1$};
\node[shape=circle,draw,fill=white,inner sep=1pt] at (1,-2) {$\scriptstyle2$};
\end{tikzpicture}
\caption{The set of successors $\mr{Suc}(1,-2)$ (black points).}
\label{fig:Suc2-1}
\end{figure}

\medskip

(i) If we have two blue points, then we get the \mes s
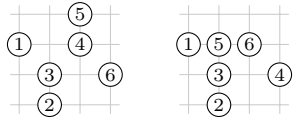
\[
\begin{tikzpicture}[scale=0.4]
\draw[color=gray!40] (-0.3,-2.3) grid (3.3,1.3);
\drawExSec{0,0;1,-2;1,-1;2,0;2,1;3,-1};
\end{tikzpicture}
\hspace{20pt}
\begin{tikzpicture}[scale=0.4]
\draw[color=gray!40] (-0.3,-2.3) grid (3.3,1.3);
\drawExSec{0,0;1,-2;1,-1;3,-1;1,0;2,0};
\end{tikzpicture}
\]
The first one is horizontally ordered. Computing the class of \mes s obtained by successively helixing to the right yields a horizontally ordered \mes\ (cf.\ the third sequence of Class (iv) of Corollary~\ref{coro:FullList} below).

\medskip

(ii) If we have two red points, then we get the \mes s
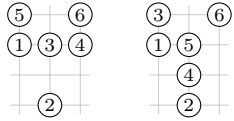
\[
\begin{tikzpicture}[scale=0.4]
\draw[color=gray!40] (-0.3,-2.3) grid (2.3,1.3);
\drawExSec{0,0;1,-2;1,0;2,0;0,1;2,1};
\end{tikzpicture}
\hspace{20pt}
\begin{tikzpicture}[scale=0.4]
\draw[color=gray!40] (-0.3,-2.3) grid (2.3,1.3);
\drawExSec{0,0;1,-2;0,1;1,-1;1,0;2,1};
\end{tikzpicture}
\]
These sequences are equivalent under helixing to the right to a sequence which can be horizontally ordered (namely the third sequence of Class (v) and the first sequence of Class (iii)' below). 

\medskip

(iii) We are left with the case of two green, one blue and one red point which leaves us with the \mes
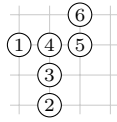
\[
\begin{tikzpicture}[scale=0.4]
\draw[color=gray!40] (-0.3,-2.3) grid (3.3,1.3);
\drawExSec{0,0;1,-2;1,-1;1,0;2,0;2,1};
\end{tikzpicture}
\]
which is horizontally ordered. This finishes the proof.
\end{proof}

Note that the involution $\sigma$ from Equation~\ref{eq:sigmaInv} on Page~\pageref{eq:sigmaInv} exchanges vertical with horizontal lex order. For a complete classification of \meset s it therefore suffices to determine the \mes s with vertical order. A straightforward, if tedious computation yields the full list of \mes s in Corollary~\ref{coro:FullList}. We subdivide this list into the equivalence classes under helixing (i)-(v). The sequences displayed here are obtained by applying $\helixR$ from the left to the right. The remaining \mes s are obtained by applying $\sigma$ since $K_X$ is fixed under $\sigma$. This yields the corresponding Classes (i)'-(v)'. We do not display those; at any rate (ii)'=(ii), (iv)'=(iv) and (v)'=(v).

\begin{corollary}[Theorem~\thmExSec\ for $X_2$]
\label{coro:FullList}
Up to applying $\sigma$, any \mes\ on $X=\P(\Omega_{\P^2}(1))$ is one of the following:

\medskip

Class (i) (the arrow indicates that we can move the middle line freely to the left or to the right; the fourth point lies by definition on $\CO_X(a,1)$): 
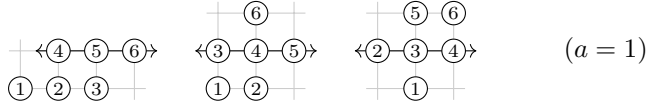
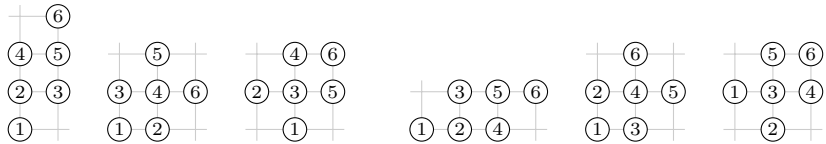
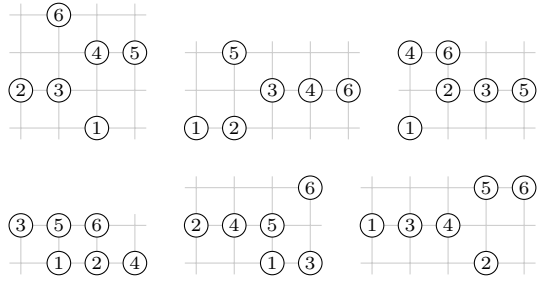
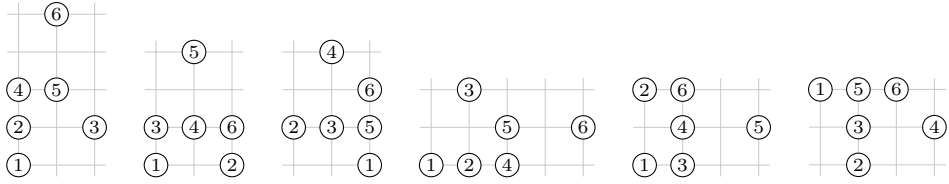
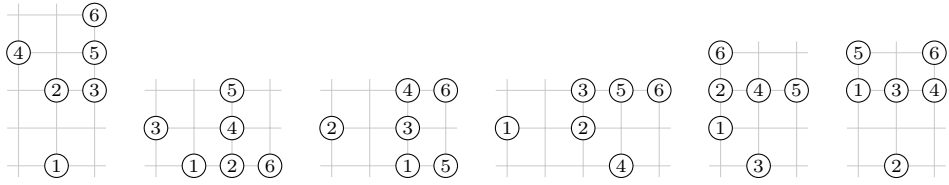
\begin{center}
\begin{tikzpicture}[scale=0.5]
\draw[color=gray!40] (-0.3,-0.3) grid (3.3,1.3);
\draw[<->] (0.4,1) -- (3.6,1);
\drawExSec{0, 0;1, 0;2, 0;1, 1;1 + 1, 1;1 + 2, 1} 
\end{tikzpicture}
\quad
\begin{tikzpicture}[scale=0.5]
\draw[color=gray!40] (-0.3,-0.3) grid (2.3,2.3);
\draw[<->] (0.4-1,1) -- (3.6-1,1);
\drawExSec{0, 0;1, 0;1 - 1, 1;1, 1;1 + 1, 1;1, 2}
\end{tikzpicture}
\quad
\begin{tikzpicture}[scale=0.5]
\draw[color=gray!40] (-1.3,-0.3) grid (1.3,2.3);
\draw[<->] (0.4-2,1) -- (3.6-2,1);
\drawExSec{0, 0;1 - 2, 1;1 - 1, 1;1, 1;0, 2;1, 2}
\draw (5,1) node {$(a=1)$};
\end{tikzpicture}
\end{center}

\medskip

Class (ii):
\begin{center}
\begin{tikzpicture}[scale=0.5]
\draw[color=gray!40] (-0.3,-0.3) grid (1.3,3.3);
\drawExSec{0, 0;0, 1;1, 1;0, 2;1, 2;1, 3}
\end{tikzpicture}
\quad
\begin{tikzpicture}[scale=0.5]
\draw[color=gray!40] (-0.3,-0.3) grid (2.3,2.3);
\drawExSec{0, 0;1, 0;0, 1;1, 1;1, 2;2, 1} 
\end{tikzpicture}
\quad
\begin{tikzpicture}[scale=0.5]
\draw[color=gray!40] (-1.3,-0.3) grid (1.3,2.3);
\drawExSec{0, 0;-1, 1;0, 1;0, 2;1, 1;1, 2}
\end{tikzpicture}
\qquad
\begin{tikzpicture}[scale=0.5]
\draw[color=gray!40] (-0.3,-0.3) grid (3.3,1.3);
\drawExSec{0, 0;1, 0;1, 1;2, 0;2, 1;3, 1} 
\end{tikzpicture}
\quad
\begin{tikzpicture}[scale=0.5]
\draw[color=gray!40] (-0.3,-0.3) grid (2.3,2.3);
\drawExSec{0, 0;0, 1;1, 0;1, 1;2, 1;1, 2} 
\end{tikzpicture}
\quad
\begin{tikzpicture}[scale=0.5]
\draw[color=gray!40] (-0.3,-1.3) grid (2.3,1.3);
\drawExSec{0, 0;1, -1;1, 0;2, 0;1, 1;2, 1}
\end{tikzpicture}
\end{center}

\medskip

Class (iii): 
\begin{center}
\begin{tikzpicture}[scale=0.5]
\draw[color=gray!40] (-2.3,-0.3) grid (1.3,3.3);
\drawExSec{0, 0;-2, 1;-1, 1;0, 2;1, 2;-1, 3}
\end{tikzpicture}
\quad
\begin{tikzpicture}[scale=0.5]
\draw[color=gray!40] (-0.3,-0.3) grid (4.3,2.3);
\drawExSec{0, 0;1, 0;2, 1;3, 1;1, 2;4, 1}
\end{tikzpicture}
\quad
\begin{tikzpicture}[scale=0.5]
\draw[color=gray!40] (-0.3,-0.3) grid (3.3,2.3);
\drawExSec{0, 0;1, 1;2, 1;0, 2;3, 1;1, 2}
\end{tikzpicture}

\bigskip

\begin{tikzpicture}[scale=0.5]
\draw[color=gray!40] (-1.3,-0.3) grid (2.3,1.3);
\drawExSec{0, 0;1, 0;-1, 1;2, 0;0, 1;1, 1}
\end{tikzpicture}
\quad
\begin{tikzpicture}[scale=0.5]
\draw[color=gray!40] (-2.3,-0.3) grid (1.3,2.3);
\drawExSec{0, 0;-2, 1;1, 0;-1, 1;0, 1;1, 2}
\end{tikzpicture}
\quad
\begin{tikzpicture}[scale=0.5]
\draw[color=gray!40] (-0.3,-1.3) grid (4.3,1.3);
\drawExSec{0, 0;3, -1;1, 0;2, 0;3, 1;4, 1}
\end{tikzpicture}
\end{center}

\medskip

Class (iv): 
\begin{center}
\begin{tikzpicture}[scale=0.5]
\draw[color=gray!40] (-0.3,-0.3) grid (2.3,4.3);
\drawExSec{0, 0;0, 1;2, 1;0, 2;1, 2;1, 4}
\end{tikzpicture}
\quad
\begin{tikzpicture}[scale=0.5]
\draw[color=gray!40] (-0.3,-0.3) grid (2.3,3.3);
\drawExSec{0, 0;2, 0;0, 1;1, 1;1, 3;2, 1}
\end{tikzpicture}
\quad
\begin{tikzpicture}[scale=0.5]
\draw[color=gray!40] (-2.3,-0.3) grid (0.3,3.3);
\drawExSec{0, 0;-2, 1;-1, 1;-1, 3;0, 1;0, 2}
\end{tikzpicture}
\quad
\begin{tikzpicture}[scale=0.5]
\draw[color=gray!40] (-0.3,-0.3) grid (4.3,2.3);
\drawExSec{0, 0;1, 0;1, 2;2, 0;2, 1;4, 1}
\end{tikzpicture}
\quad
\begin{tikzpicture}[scale=0.5]
\draw[color=gray!40] (-0.3,-0.3) grid (3.3,2.3);
\drawExSec{0, 0;0, 2;1, 0;1, 1;3, 1;1, 2}
\end{tikzpicture}
\quad
\begin{tikzpicture}[scale=0.5]
\draw[color=gray!40] (-0.3,-2.3) grid (3.3,0.3);
\drawExSec{0, 0;1, -2;1, -1;3, -1;1, 0;2, 0}
\end{tikzpicture}
\end{center}

\medskip

Class (v): 
\begin{center}
\begin{tikzpicture}[scale=0.5]
\draw[color=gray!40] (-1.3,-0.3) grid (1.3,4.3);
\drawExSec{0, 0;0, 2;1, 2;-1, 3;1, 3;1, 4}
\end{tikzpicture}
\quad
\begin{tikzpicture}[scale=0.5]
\draw[color=gray!40] (-1.3,-0.3) grid (2.3,2.3);
\drawExSec{0, 0;1, 0;-1, 1;1, 1;1, 2;2, 0}
\end{tikzpicture}
\quad
\begin{tikzpicture}[scale=0.5]
\draw[color=gray!40] (-2.3,-0.3) grid (1.3,2.3);
\drawExSec{0, 0;-2, 1;0, 1;0, 2;1, 0;1, 2}
\end{tikzpicture}
\quad
\begin{tikzpicture}[scale=0.5]
\draw[color=gray!40] (-0.3,-1.3) grid (4.3,1.3);
\drawExSec{0, 0;2, 0;2, 1;3, -1;3, 1;4, 1}
\end{tikzpicture}
\quad
\begin{tikzpicture}[scale=0.5]
\draw[color=gray!40] (-0.3,-1.3) grid (2.3,2.3);
\drawExSec{0, 0;0, 1;1, -1;1, 1;2, 1;0, 2}
\end{tikzpicture}
\quad
\begin{tikzpicture}[scale=0.5]
\draw[color=gray!40] (-0.3,-2.3) grid (2.3,1.3);
\drawExSec{0, 0;1, -2;1, 0;2, 0;0, 1;2, 1}
\end{tikzpicture}
\end{center}
\end{corollary}

From Corollary~\ref{coro:FullList} we can immediately compute $\pF(\cl)$ via Equation~\eqref{eq:CompF} on Page~\pageref{eq:CompF}. Since we always have $\Eff(\sigma)\subseteq\pF(\cl)$, we merely compute the complement 
\[
\pF_0(\sigma):=\pF(\cl)\setminus\Eff(\sigma).
\]
Moreover, we only give the stable range for the sequences of Class (i).

\begin{corollary}[Theorem~\thmStrongExSec\ for $X_2$]
\label{cor:StrExSecX2}
The partial order associated with the \mes s of Corollary~\ref{coro:FullList} have the following generating sets:

\medskip

Class (i): For the first sequence, we find
\[
\pF_0(\cl)=\left\{
\begin{array}{ll}
\varnothing,&\quad a\geq1\\
\{(\cl_i,\cl_j)\mid1\leq i\leq 3;\;4\leq j\leq 6\},&\quad a\leq-6.
\end{array}\right.
\]
For the second sequence, we find
\[
\pF_0(\cl)=\left\{
\begin{array}{ll}
\{(\cl_3,\cl_6),\,(\cl_4,\cl_6),\,(\cl_5,\cl_6)\},&\quad a\geq5\\
\varnothing,&\quad a=1\\
\{(\cl_i,\cl_j)\mid i=1,\,2;\;3\leq j\leq 5\},&\quad a\leq-4.
\end{array}\right.
\]
For the third sequence, we find
\[
\pF_0(\cl)=\left\{
\begin{array}{ll}
\{(\cl_i,\cl_j)\mid 2\leq i\leq 4;\;j=5,\,6\},&\quad a\geq7\\
\varnothing,&\quad a=1\\
\{(\cl_1,\cl_2),\,(\cl_1,\cl_3),\,(\cl_1,\cl_4)\},&\quad a\leq-4.
\end{array}\right.
\]
The first sequence is strongly exceptional if $a\geq1$, whereas the other sequences are strongly exceptional only if $a=1$. In fact, they are all strongly cyclic for $a=1$.

\medskip

Class (ii):
$\pF_0(\cl)=\varnothing$. Thus, all sequences are strongly exceptional. In particular, this class is strongly cyclic.

\medskip

Class (iii):
\begin{align*}
\pF_0(\cl)=
&\{(\cl_2,\cl_3),\,(\cl_4,\cl_5)\}\,|\,\{(\cl_1,\cl_2),\,(\cl_3,\cl_4)\}\,|\,\{(\cl_2,\cl_3)\}\\
&\{(\cl_1,\cl_2),\,(\cl_5,\cl_6)\}\,|\,\{(\cl_4,\cl_5)\}\,|\,\{(\cl_3,\cl_4),\,(\cl_5,\cl_6)\}.
\end{align*}

\medskip

Class (iv): 
\begin{align*}
\pF_0(\cl)=&\{(\cl_3,\cl_4)\}\,|\,\{(\cl_2,\cl_3),\,(\cl_5,\cl_6)\}\,|\,\{(\cl_1,\cl_2),\,(\cl_4,\cl_5)\}\\
&\{(\cl_3,\cl_4)\}\,|\,\{(\cl_2,\cl_3),\,(\cl_5,\cl_6)\}\,|\,\{(\cl_1,\cl_2),\,(\cl_4,\cl_5)\}.
\end{align*}

\bigskip

Class (v): 
\begin{align*}
\pF_0(\cl)=&\{(\cl_3,\cl_4)\}\,|\,\{(\cl_2,\cl_3),\,(\cl_5,\cl_6)\}\,|\,
\{(\cl_1,\cl_2),\,(\cl_4,\cl_5)\}\\
&\{(\cl_3,\cl_4)\}\,|\,
\{(\cl_2,\cl_3),\,(\cl_5,\cl_6)\}\,|\,\{(\cl_1,\cl_2),\,(\cl_4,\cl_5)\}. 
\end{align*}
\end{corollary}

\subsection{Fullness of \meset s on $X_2$}
\label{subsec:FullXl}
We are now in a position to prove Theorem~\thmFull.

\begin{theorem}[Theorem~\thmFull\ for $X_2$]
\label{thm:CforX2}
On $X=\P(\Omega_{\P^2}(1))$ any \mes\ can be mutated to a sequence of Orlov-type. In particular, any \meset\ is full.
\end{theorem}

\begin{proof}
We only need to check the statement for a single exceptional sequence in each of the helix classes listed in Corollary~\ref{coro:FullList}. 

\medskip

Case (i) is obvious by considering the first sequence. For Case (ii) we take the third, and for Case (iii) the fourth sequence, where we have the mutually orthogonal pairs $((1,-1),(0,2))$ and $((1,1),(0,2))$. Changing the order immediately reveals that these sequences are of Orlov-type, too.

\medskip

Finally, after tensoring with a suitable line bundle we can appeal to Lemma~\ref{lem:banana} below and mutate the fourth sequences of Cases (iv) and (v) into exceptional sequences of Orlov-type as follows:
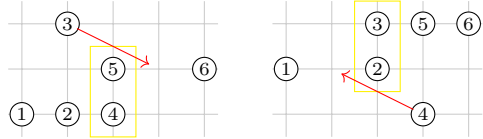
\[
\begin{tikzpicture}[scale=0.6]
\draw[color=gray!40] (-0.3,-0.5) grid (4.3,2.5);
\drawExSec{0,0;1,0;1,2;2,0;2,1;4,1}
\draw[color=yellow] (1.5,-0.5) rectangle (2.5,1.5);
\draw[->,shorten <=4pt,shorten >=4pt,color=red] (1,2) -- (3,1);
\end{tikzpicture}
\quad\quad
\begin{tikzpicture}[scale=0.6]
\draw[color=gray!40] (-0.3,-1.5) grid (4.3,1.5);
\drawExSec{0,0;2,0;2,1;3,-1;3,1;4,1}
\draw[color=yellow] (1.5,-0.5) rectangle (2.5,1.5);
\draw[->,shorten <=4pt,shorten >=4pt,color=red] (3,-1) -- (1,0);
\end{tikzpicture}
\qedhere
\]
\end{proof}

It remains to prove the

\begin{lemma}
\label{lem:banana}
On $X = \P(\Omega_{\P^2}(1))$, the two (non full) exceptional sequences
\[
\big((-2,1), (-1,-1),(-1,0))\longleftrightarrow((-1,-1),(-1,0),(0,0)\big)
\]
are related by mutating $(-2,1)$ to the right, or alternatively, mutating $(0,0)$ to the left. More precisely, $\CO(-2h+H)[-1] = \leftMut_{\CO(-h-H)} \leftMut_{\CO(-h)} \CO$ where we ignore the shift in the exceptional sequences above.
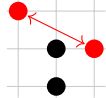
\begin{center}
\begin{tikzpicture}[scale=0.5]
\draw[color=gray!40] (-2.3,-1.3) grid (0.3,1.3);
 \fill[thick, color=black] (-1,0) circle (7pt);
 \fill[thick, color=black] (-1,-1) circle (7pt);
 \fill[thick, color=red] (-2,1) circle (7pt);
 \fill[thick, color=red] (0,0) circle (7pt);
\draw[<->,shorten <=4pt,shorten >=4pt,color=red] (-2,1) -- (0,0);
\end{tikzpicture}
\end{center}
Dually, the two (non full) exceptional sequences
\[
\big((0,0),(1,0),(1,1))\longleftrightarrow((1,0), (1,1),(2,-1)\big)
\]
are related by mutating $(0,0)$ to the right, or alternatively, mutating $(2,-1)$ to the left. More precisely, $\CO(2h-H)[1] = \rightMut_{\CO(h+H)} \rightMut_{\CO(h)} \CO$
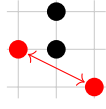
\begin{center}
\begin{tikzpicture}[scale=0.5]
\draw[color=gray!40] (-0.3,-1.3) grid (2.3,1.3);
 \fill[thick, color=black] (1,0) circle (7pt);
 \fill[thick, color=black] (1,1) circle (7pt);
 \fill[thick, color=red] (2,-1) circle (7pt);
 \fill[thick, color=red] (0,0) circle (7pt);
\draw[<->,shorten <=4pt,shorten >=4pt,color=red] (2,-1) -- (0,0);
\end{tikzpicture}
\end{center}
\end{lemma}

\begin{proof}
To ease notation we use $X\cong \P(\Tang_{\P^2})$ as $\Tang_{\P^2} \cong \Omega_{\P^2}(-3)$, and we write $\Tang$ and $\Omega$ for $\Tang_{\P^2}$ and $\Omega_{\P^2}$, respectively. In particular, we find $H = 2h+H'$ for the relative hyperplane section $H'$ of $\P(\Tang)$. 

\medskip

Recall that there is a tautological morphism $\phi\colon\pi^*\Omega\to \CO(H')$ which yields the short exact sequence
\[
0\to\ker\phi\to\pi^*\Omega\to\CO(H')\to0.
\]
To compute $\ker\phi=\CO(ih+jH')$ we use $c(\Tang) = (1+h)^3 = 1 +3h+3h^2 \in A(\P^2) = \Z[ht]/h^3$ and~\cite[Example 8.3.4]{fultonIntersection}. We obtain
\[
A\big(\P(\Tang)\big)=\Z[h,H']/(h^3,(H')^2+3hH'+3h^2).
\]
In particular, $c(\ker\phi)=c(\pi^*\Omega)/c(\CO(H'))$, or more explicitly,
\[
1+ih+jH' = (1-3h+3h^2)/(1+H') = 1-3h-H'
\]
in $A\big(\P(\Tang)\big)$. As a result, we arrive at the short exact sequence
\[
0\to\CO(-3h-H')\to\pi^*\Omega\to\CO(H')\to0.
\]
With respect to the nef basis $(h,H)$ of $X = \P(\Omega(1))$ on $\Pic(X)$ this is
\begin{equation}
\label{eq:banana}
0\to \CO(-h-H) \to \pi^* \Omega \to \CO(-2h+H) \to 0,
\end{equation}
which can be also written as the triangle
\[
\CO(-2h+H)[-1] \to \CO(-h-H) \to \pi^* \Omega.
\]
One checks that $\Hom^\bullet(\CO(-2h+H),\pi^*\Omega) = 0$. Therefore, the triangle implies that 
\[
\Hom^\bullet(\CO(-h-H),\pi^*\Omega) \cong \Hom^\bullet(\pi^*\Omega,\pi^*\Omega) \cong \C
\]
whence $\leftMut_{\CO(-h-H)}\pi^*\Omega=\CO(-2h+H)[-1]$.

\medskip

On the other hand, $\pi^*\Omega$ also fits into the pullback of the Euler sequence
\[
0 \to \pi^* \Omega \to \CO(-h)^{\oplus 3} \to \CO \to 0,
\]
so $\leftMut_{\CO(-h)}\CO=\pi^*\Omega$. Putting everything together, we can carry out the following sequence of mutations:
\[
\begin{split}
(\CO(-h-H),\CO(-h),\CO) 
&\!\leadsto\!(\CO(-h-H), \pi^* \Omega = \leftMut_{\CO(-h)} \CO, \CO(-h))\\
&\!\leadsto\!(\CO(-h+H)[-1] = \leftMut_{\CO(-h-H)} \pi^* \Omega, \CO(-h-H),\CO(-h))
\end{split}
\]
In other words, $\CO(-2h+H)[-1] = \leftMut_{\CO(-h-H)} \leftMut_{\CO(-h)} \CO$. The second case follows dually.
\end{proof}

\section{Some remarks on the Rouquier dimension}
\label{sec:Rouquier}
Given a tilting object $T$ in a triangulated category $\mc T$ we can consider the following ascending chain of (not necessarily triangulated) subcategories of $\mc T$:
\begin{enumerate}
\item $\langle T \rangle_0 = \mathrm{add}(T)$ the smallest full subcategory which (a) contains the summands of $T$ and (b) is closed under direct sums, shifts and isomorphisms.
\item $\langle T \rangle_i$ the smallest full subcategory which (a) is closed under direct sums and shifts and (b) contains all objects $X$ that fit into triangles $A \to B \to X \oplus X'$ with $A \in \langle T \rangle_{i-1}$, $B \in \langle T \rangle_0$, and $X'$ an object complementing $X$ to the cone of $A$ and $B$.
\end{enumerate}
The {\em generation time} of $T$ is the smallest integer $i\in\N$ such that $\mc T = \langle T \rangle_i$. The {\em Rouquier dimension} of $\mc T$ is the minimum over the generation times of all tilting objects in $\mc T$.

\begin{theorem}
Let $X$ be a smooth projective variety. 
\begin{enumerate}
\item[\rm(i)] $\dim(X)$ is smaller or equal to the Rouquier dimension of $\CD(X)$~\cite{Rouquier}.
\item[\rm(ii)] For a given tilting object $T\in\CD(X)$ let $i_0$ be the largest $i$ such that 
\[
\R^i\!\operatorname{Hom}(T,T \otimes K_X^{-1})=\Hom_{\CD(X)}(T,T \otimes K_X^{-1}[i]) \neq 0.
\]
Then the generation time of $T$ is equal to $\dim(X)+i_0$~\cite{BF12}.
\end{enumerate}
In particular, the Rouquier dimension of $\CD(X)$ is bounded by $\dim(X)+i_0$.
\end{theorem}

\begin{example}
\label{exam:rouquier_upperBound}
Let $\pi\colon X=\P(\CE)\to\P^\lH$ be a $\P^\lV$-bundle. Replacing $\CE$ by $\CE(D)$ for some divisor $D$ if necessary, we may assume that the relative hyperplane section $H$ is sufficiently positive. This implies that the full exceptional sequence of Orlov-type
\begin{equation}
\label{eq:TiltingBundle}
\big(\CO, \CO(h),\ldots,\CO(\lH h),\CO(h+H),\ldots,\CO(\lH h + \lV H)\big)\end{equation}
is strongly exceptional. Let $T$ be the tilting object given by the direct sum of these line bundles. Since
\begin{align*}
&\;\R^k\!\Hom\big(\CO_X(i_1h +j_1H),\CO_X(i_2h+j_2H)\otimes K_X^{-1}\big) \\
=&\;\mr{H}^k\big(\P^\lH,\CO_{\P^\lH}(i_2-i_1)\otimes\R\pi_* \CO_X\big((j_2-j_1)H\otimes K_X^{-1}\big)\big)
\end{align*}
and $\R \pi_* \CO_X\big((j_2-j_1)H \otimes K_X^{-1}\big)$ is of the form $\CO(i')\otimes\Sym^j\CE^\vee$ and thus a bundle on $\P^\lH$, we have 
$\R^i\!\Hom(T,T\otimes K_X^{-1}) = 0$ for $i > \lH$. Hence the generation time of $T$ is bounded by $\dim(X)+\lH$.
\end{example}

A conjecture of Orlov states that the Rouquier dimension of $\CD(X)$ actually coincides with $\dim(X)$. We can sharpen the previous estimate in order to support this conjecture in the following two cases:

\begin{proposition}
\label{prop:rouquier}
The Rouquier dimension of $\CD(X)$ is equal to $\dim(X)$ if 
\begin{enumerate}
\item[\rm(i)] $X=X(\lH,\lV;c)$ and $-K_X$ is nef.
\item[\rm(ii)] $X=X_\lH$.
\end{enumerate}
\end{proposition}

\begin{proof}
We consider the tilting bundle from Example~\ref{exam:rouquier_upperBound}.

\medskip

(i) Since $-K_X = (\lH+1-\beta)h+(\lV+1)H$, we have
\begin{align*}
&\;\mr H^k\big(X,\CO_X((i_2-i_1)h + (j_2-j_1)H) \otimes K_X^{-1}\big)\\
=&\;\mr H^k\big(X,\CO_X((i_2-i_1+\lH+1-\beta)h + \underbrace{(j_2-j_1+ \lV+1)}_{>0}H)\big).
\end{align*}
In order for this to vanish, we gather from the shape of the immaculate locus that the only higher cohomology group to avoid is $H^\lH$. This requires $i_2-i_1+\lH+1-\beta > -\lH-1$, or equivalently, $\beta \leq i_2-i_1 + 2\lH+1$. As $i_2-i_1\geq - \lH$, the $\lH$-th cohomology vanishes if and only if $ \beta \leq \lH + 1$, that is, $\lH+1-\beta \geq 0$ which is tantamount to saying that $-K_X$ is nef.

\medskip

(ii) Since $K_{X_\lH}=\CO(-\lH, -\lH)$, $\Hom^\bullet(T,T\otimes\omega^{-1})$ is concentrated in degree zero. 
\end{proof}

\begin{remark}
In the recent article~\cite{RouquDim}, Orlov's conjecture has been established for any toric variety $X$. In particular, the nefness of $-K_X$ in Proposition~\ref{prop:rouquier} (i) is not a necessary condition. 
\end{remark}

\begin{remark}
For $X_\lH^\vee:=\P(\Tang_{\P^{\lH}}(2))$ the proof above does not work if $\lH \geq 3$. Indeed, consider the exceptional sequences of Orlov-type
\[
\begin{split}
( &(0,0), (1,0), \ldots, (\lH,0), \\
  &(a_1,1), (a_1+1,1), \ldots, (a_1+\lH,1), \\
  & \ldots \\
  &(a_{\lH-1},\lH-1), \ldots, (a_{\lH-1}+\lH,\lH-1) )
\end{split}
\]
for some integers $a_0=0,a_1,\ldots,a_{\lH-1}$. Such a sequence is strong if and only if $a_{i+1}-a_i\geq\lH-1$ for all $i$. In particular, $a_{\lH-1} \geq (\lH-1)^2$.

\medskip

Using that $K_{X_\lH^\vee}=(-2,-\lH)$, we see that
\[
\Hom^\bullet\big(\CO((a_{\lH-1}+\lH)h+(\lH-1)H),\CO \otimes \omega^{-1}\big)= 
H^\bullet\big(\CO((2-a_{\lH-1}-\lH)h+H)\big).
\]
Note that $2-a_{\lH-1}-\lH \leq -\lH^2-\lH+3 \leq -3\lH$ if $\lH\geq 3$. By Corollary~\ref{coro:cohomPOmega},
\[
H^\lH\big(\CO((2-a_{\lH-1}-\lH)h+H)\big)\neq0 
\]
for $\lH\geq3$, which merely yields $\dim(X_\lH)+\lH$ as an upper bound for the Rouquier dimension.

\medskip

Although we did not check all full exceptional sequences of line bundles on $X_\lH$, we believe that one cannot obtain a lower bound for the Rouquier dimension using those as soon as $\lH\geq3$. Since Orlov's conjecture is commonly supposed to be true one has to look for a more sophisticated way of producing tilting bundles.
\end{remark}


\begin{thebibliography}{ABKW20}

\bibitem[AA22]{p1p1p1}
Klaus Altmann and Martin Altmann.
\newblock Exceptional sequences of 8 line bundles on {{\((\mathbb{P}^1 )^3\)}}.
\newblock {\em J. Algebr. Comb.}, 56(2):305--322, 2022.
\newblock \arXiv{2108.11806}.

\bibitem[ABKW20]{immaculate}
Klaus Altmann, Jaros{\l}aw Buczy\'{n}ski, Lars Kastner, and Anna-Lena Winz.
\newblock Immaculate line bundles on toric varieties.
\newblock {\em Pure Appl. Math. Q.}, 16(4):1147--1217, 2020.
\newblock \arXiv{1808.09312}.

\bibitem[AW24]{fesfano}
Klaus Altmann and Frederik Witt.
\newblock The structure of exceptional sequences on toric varieties of {Picard}
  rank two.
\newblock {\em Algebr. Comb.}, 7(4):1039--1074, 2024.
\newblock \arXiv{2112.14637}.

\bibitem[{Bei}78]{beilinson}
A.A. {Beilinson}.
\newblock {Coherent sheaves on {$\mathbb{P}^n$} and problems in linear
  algebra.}
\newblock {\em {Funktsional. Anal. i Prilozhen.}}, 12(3):68--69, 1978.

\bibitem[BF12]{BF12}
Matthew Ballard and David Favero.
\newblock Hochschild dimensions of tilting objects.
\newblock {\em Int. Math. Res. Not. IMRN}, (11):2607--2645, 2012.
\newblock \arXiv{0905.1444}.

\bibitem[BGKS15]{Phantom}
Christian {B\"ohning}, Hans-Christian {Graf von Bothmer}, Ludmil {Katzarkov},
  and Pawel {Sosna}.
\newblock {Determinantal Barlow surfaces and phantom categories.}
\newblock {\em {J. Eur. Math. Soc. (JEMS)}}, 17(7):1569--1592, 2015.
\newblock \arXiv{1210.0343}.

\bibitem[Bri05]{brionFlag}
Michel Brion.
\newblock Lectures on the geometry of flag varieties.
\newblock In {\em Topics in cohomological studies of algebraic varieties.
  Impanga lecture notes}, pages 33--85. Birkh{\"a}user, 2005.
\newblock \arXiv{math/0410240}.

\bibitem[BW21]{borisov21}
Lev Borisov and Chengxi Wang.
\newblock On strong exceptional collections of line bundles of maximal length
  on {Fano} toric {Deligne}-{Mumford} stacks.
\newblock {\em Asian J. Math.}, 25(4):505--520, 2021.
\newblock \arXiv{1907.01135}.

\bibitem[FH23]{RouquDim}
David {Favero} and Jesse {Huang}.
\newblock {Rouquier dimension is Krull dimension for normal toric varieties}.
\newblock {\em {Eur. J. Math.}}, 9:91, 2023.
\newblock \arXiv{2302.09158}.

\bibitem[Ful84]{fulton}
William Fulton.
\newblock {\em Intersection theory}, volume~2 of {\em Ergeb. Math. Grenzgeb.,
  3. Folge}.
\newblock Springer, Cham, 1984.

\bibitem[Ful98]{fultonIntersection}
William Fulton.
\newblock {\em Intersection theory}, volume~2 of {\em Ergebnisse der Mathematik
  und ihrer Grenzgebiete}.
\newblock Springer-Verlag, 1998.

\bibitem[{Kle}88]{kle88}
Peter {Kleinschmidt}.
\newblock {A classification of toric varieties with few generators}.
\newblock {\em {Aequationes Math.}}, 35(2-3):254--266, 1988.

\bibitem[Kra24]{Krah}
Johannes Krah.
\newblock A phantom on a rational surface.
\newblock {\em Invent. Math.}, 235(3):1009--1018, 2024.
\newblock \arXiv{2304.01269}.

\bibitem[{Kuz}14]{KuznetsovICM}
Alexander {Kuznetsov}.
\newblock {Semiorthogonal decompositions in algebraic geometry}.
\newblock In {\em Proceedings of the International Congress of Mathematicians
  (ICM 2014), Vol. II}, pages 635--660. Seoul: KM Kyung Moon Sa, 2014.
\newblock \arXiv{1404.3143}.

\bibitem[Laz04]{lazarsfeld-bundles}
Robert Lazarsfeld.
\newblock {\em Positivity in algebraic geometry. {II}}, volume~49 of {\em
  Ergebnisse der Mathematik und ihrer Grenzgebiete}.
\newblock Springer-Verlag, 2004.

\bibitem[{Lui}]{luisFlag}
{Luis Sol\'a Conde}.
\newblock {Geometry of rational homogeneous spaces}.
\newblock Manuscript \L{}uk\k ecin September 9-13, 2019.

\bibitem[{Orl}92]{orlov92}
Dmitri {Orlov}.
\newblock {Projective bundles, monoidal transformations, and derived categories
  of coherent sheaves.}
\newblock {\em {Izv. Math.}}, 41(1):133--141, 1992.

\bibitem[Orl09]{Orlov-generators}
Dmitri Orlov.
\newblock Remarks on generators and dimensions of triangulated categories.
\newblock {\em Mosc. Math. J.}, 9(1):153--159, 2009.
\newblock \arXiv{0804.1163}.

\bibitem[Rou08]{Rouquier}
Rapha\"{e}l Rouquier.
\newblock Dimensions of triangulated categories.
\newblock {\em J. K-Theory}, 1(2):193--256, 2008.
\newblock \arXiv{math/0310134}.

\bibitem[Rud90]{Rudakov}
Alexei Rudakov, editor.
\newblock {\em Helices and vector bundles: seminaire {Rudakov}}, volume 148 of
  {\em Lond. Math. Soc. Lect. Note Ser.}
\newblock Cambridge University Press, 1990.

\end{thebibliography}
\end{document}